\documentclass[hidelinks,onefignum,onetabnum]{siamart220329}

\usepackage{import}

\usepackage{algorithm}
\usepackage{amsmath,amssymb} 
\usepackage[english]{babel} 
\usepackage{bm}
\usepackage{booktabs}    
\usepackage{comment} 
\usepackage{empheq}
\usepackage[inline]{enumitem}
\usepackage{fancyhdr}
\usepackage[T1]{fontenc}    
\usepackage{gensymb}
\usepackage[utf8]{inputenc} 
\usepackage{latexsym}
\usepackage{listings} 
\usepackage{mathtools}
\usepackage{multicol,multirow}
\usepackage{nicefrac}  
\usepackage{subcaption}
\usepackage{times}
\usepackage{thmtools}
\usepackage{upquote}
\usepackage{url}            
\usepackage{wrapfig}
\usepackage{xcolor}
\usepackage{cleveref}

\usepackage[normalem]{ulem}

\newcommand{\ignore}[1]{}

\newtheorem{example}{Example}[section]

\newcommand{\C}{\mathbb{C}} 

\newcommand{\E}{\mathbb{E}}

\newcommand{\N}{\mathbb{N}} 
\newcommand{\PP}{\mathbb{P}} 
\newcommand{\R}{\mathbb{R}} 
\newcommand{\SSS}{\mathbb{S}}

\newcommand{\ub}{\mathbf{u}}
\newcommand{\vb}{\mathbf{v}}

\newcommand{\xb}{\mathbf{x}}
\newcommand{\yb}{\mathbf{y}}

\newcommand{\Ab}{\mathbf{A}}
\newcommand{\Bb}{\mathbf{B}}

\newcommand{\Eb}{\mathbf{E}}

\newcommand{\Gb}{\mathbf{G}}

\newcommand{\Ib}{\mathbf{I}}

\newcommand{\Mb}{\mathbf{M}}
\newcommand{\Nb}{\mathbf{N}}

\newcommand{\Pb}{\mathbf{P}}
\newcommand{\Qb}{\mathbf{Q}}
\newcommand{\Rb}{\mathbf{R}}

\newcommand{\Ub}{\mathbf{U}}
\newcommand{\Vb}{\mathbf{V}}
\newcommand{\Wb}{\mathbf{W}}
\newcommand{\Xb}{\mathbf{X}}
\newcommand{\Yb}{\mathbf{Y}}

\newcommand{\Ncal}{\mathcal{N}}

\newcommand{\Ucal}{\mathcal{U}}
\newcommand{\Vcal}{\mathcal{V}}

\newcommand{\Xcal}{\mathcal{X}}

\newcommand{\eps}{\epsilon}

\newcommand{\betab}{\boldsymbol{\beta}}

\newcommand{\nub}{\boldsymbol{\nu}}

\newcommand{\sigmab}{\boldsymbol{\sigma}}

\newcommand{\Sigmab}{\boldsymbol{\Sigma}}

\newcommand{\Omegab}{\boldsymbol{\Omega}}

\newcommand{\dfeq}{\triangleq}
\newcommand{\aleq}{\preccurlyeq}
\newcommand{\ageq}{\succcurlyeq}
\newcommand{\pinv}{\dagger}

\newcommand*{\argmin}{\mathop{\mathrm{argmin}}}
\newcommand*{\argmax}{\mathop{\mathrm{argmax}}}

\newcommand{\tr}{\mathop{\mathrm{tr}}}
\newcommand{\diag}{\mathop{\mathrm{diag}}}
\newcommand{\rank}{\mathop{\mathrm{rank}}}
\newcommand{\range}{\mathop{\mathrm{Range}}}
\newcommand{\row}{\mathop{\mathrm{Row}}}
\newcommand{\col}{\mathop{\mathrm{Col}}}

\newcommand{\spn}{\mathop{\mathrm{span}}}
\newcommand{\ortho}{\mathop{\mathrm{ortho}}}

\newcommand{\wh}[1]{\widehat{#1}}
\newcommand{\wt}[1]{\widetilde{#1}}

\newcommand{\eg}{\emph{e.g.}}
\newcommand{\ie}{\emph{i.e.}}
\newcommand{\iid}{\emph{i.i.d.}}
\newcommand{\cf}{\emph{cf.}\ }

\newcommand{\rbr}[1]{\left(#1\right)}
\newcommand{\sbr}[1]{\left[#1\right]}
\newcommand{\cbr}[1]{\left\{#1\right\}}
\newcommand{\nbr}[1]{\left\|#1\right\|}
\newcommand{\abbr}[1]{\left|#1\right|}

\newcommand{\nnbr}[1]{{\left\vert\kern-0.25ex\left\vert\kern-0.25ex\left\vert #1 \right\vert\kern-0.25ex\right\vert\kern-0.25ex\right\vert}}

\newcommand{\csepp}[2]{\left\{#1 ~\middle|~ #2 \right\}}

\newcommand{\norm}[1]{\|#1\|}

\newcommand{\bmat}[1]{\begin{bmatrix} #1 \end{bmatrix}}

\definecolor{cyan}{rgb}{0.0, 1.0, 1.0}
\definecolor{magenta}{rgb}{0.79, 0.08, 0.48}
\definecolor{cssgreen}{rgb}{0.0, 0.5, 0.0}

\renewcommand{\b}{\textbf}
\renewcommand{\t}[1]{\text{#1}}

\usepackage{lipsum}
\usepackage{amsfonts}
\usepackage{graphicx}
\usepackage{epstopdf}
\usepackage{algorithmic}
\ifpdf
  \DeclareGraphicsExtensions{.eps,.pdf,.png,.jpg}
\else
  \DeclareGraphicsExtensions{.eps}
\fi

\newsiamremark{remark}{Remark}
\newsiamremark{hypothesis}{Hypothesis}
\crefname{hypothesis}{Hypothesis}{Hypotheses}
\newsiamthm{claim}{Claim}

\headers{Randomized Subspace Approximations}{Y. Dong, P. G. Martinsson, and Y. Nakatsukasa}

\title{Efficient Bounds and Estimates for Canonical Angles in Randomized Subspace Approximations\thanks{
  Submitted to the editors \today.
}}

\author{
  Yijun Dong\thanks{Courant Institute, New York University (\email{yd1319@nyu.edu}).}
  \and Per-Gunnar Martinsson\thanks{Department of Mathematics \& Oden Institute, University of Texas at Austin (\email{pgm@oden.utexas.edu}).}
  \and Yuji Nakatsukasa\thanks{Mathematical Institute, University of Oxford (\email{nakatsukasa@maths.ox.ac.uk}).}
}

\usepackage{amsopn}

\ifpdf
\hypersetup{
  pdftitle={Efficient Bounds and Estimates for Canonical Angles in Randomized Subspace Approximations},
  pdfauthor={Y. Dong, P. G. Martinsson, and Y. Nakatsukasa}
}
\fi

\begin{document}

\maketitle

\begin{abstract}
    Randomized subspace approximation with ``matrix sketching'' is an effective approach for constructing approximate partial singular value decompositions (SVDs) of large matrices. 
    The performance of such techniques has been extensively analyzed, and very precise estimates on the distribution of the residual errors have been derived.
    However, our understanding of the accuracy of the computed singular vectors (measured in terms of the canonical angles between the spaces spanned by the exact and the computed singular vectors, respectively) remains relatively limited.
    In this work, we present {\emph{practical bounds and estimates for canonical angles} of randomized subspace approximation that \emph{can be computed efficiently either a priori or a posteriori}, without assuming prior knowledge of the true singular subspaces}. 
    Under moderate oversampling in the randomized SVD, our prior probabilistic bounds are asymptotically tight and can be computed efficiently, while bringing a clear insight into the balance between oversampling and power iterations given a fixed budget on the number of matrix-vector multiplications.
    The numerical experiments demonstrate the empirical effectiveness of these canonical angle bounds and estimates on different matrices under various algorithmic choices for the randomized SVD.
\end{abstract}  



\section{Introduction}

In light of the ubiquity of high-dimensional data in modern computation, dimension reduction tools like low-rank matrix {approximations} are becoming indispensable tools for managing large data sets. In general, the goal of low-rank matrix {approximation} is to identify bases of proper low-dimensional subspaces that well encapsulate the dominant components in the original column and row spaces.
As one of the most well-established forms of matrix decompositions, the truncated singular value decomposition (SVD) is known to achieve the optimal low-rank approximation errors for any given ranks \cite{eckart1936}. Moreover, the corresponding left and right leading singular subspaces can be broadly leveraged for problems like principal component analysis, canonical correlation analysis, spectral clustering \cite{boutsidis2015spectral}, and leverage score sampling for matrix skeleton selection \cite{drineas2008relative,mahoney2009,drineas2012}.

However, for large matrices, the computational cost of classical algorithms for computing the SVD (\cf \cite[Lec.~31]{trefethen1997} or \cite[Sec.~8.6.3]{golub2013}) quickly becomes prohibitive. Fortunately, a randomization framework known as ``matrix sketching'' \cite{halko2011finding,woodruff2015sketching} provides a simple yet effective remedy for this challenge by {projecting} large matrices to random low-dimensional subspaces where the classical SVD algorithms can be executed efficiently. 

Concretely, for an input matrix $\Ab \in \C^{m \times n}$ and a target rank $k \ll \min(m,n)$, the basic version of the randomized SVD \cite[Alg.~4.1]{halko2011finding} starts by drawing a Gaussian random matrix $\Omegab \in \C^{n \times l}$ for a sample size $l$ that is slightly larger than $k$ so that $k < l \ll \min(m,n)$. Then through a matrix-matrix multiplication $\Xb = \Ab\Omegab$ with $O(mnl)$ complexity, the $n$-dimensional row space of $\Ab$ is {projected} to a random $l$-dimensional subspace. With the low-dimensional range approximation $\Xb$, a rank-$l$ randomized SVD $\wh\Ab_l = \wh\Ub_l \wh\Sigmab_l \wh\Vb_l^*$ can be constructed efficiently by computing the QR and SVD of small matrices in $O\rbr{(m+n)l^2}$ time. When the spectral decay in $\Ab$ is slow, a few power iterations $\Xb=\rbr{\Ab\Ab^*}^q\Ab\Omegab$ (usually $q=1,2$) can be incorporated to enhance the accuracy, 
\cf\cite{halko2011finding} Algorithms 4.3 and 4.4.

Let 
$\Ab=\Ub\Sigmab\Vb^*$\footnote{
Here $\Ub \in \C^{m \times r}$, $\Vb \in \C^{n \times r}$, and $\Sigmab \in \C^{r \times r}$. $\Sigmab$ is a diagonal matrix with positive non-increasing diagonal entries, and $r \le \min(m,n)$.
} denote the (unknown) full SVD of $\Ab$. In this work, we explore the alignment between the true leading rank-$k$ singular subspaces $\Ub_k,\Vb_k$ and their respective rank-$l$ approximations $\wh\Ub_l, \wh\Vb_l$ in terms of the canonical angles $\angle\rbr{\Ub_k,\wh\Ub_l}$ and $\angle\rbr{\Vb_k,\wh\Vb_l}$. We introduce prior statistical guarantees and unbiased estimates for these angles with respect to $\Sigmab$, as well as posterior deterministic bounds with additional dependence on $\wh\Ab_l = \wh\Ub_l \wh\Sigmab_l \wh\Vb_l^*$. {In particular, we are interested in \emph{practical bounds and estimates that can be computed efficiently}, without assuming prior knowledge of the true singular subspaces.}

\subsection{Our Contributions}
\paragraph{Prior probabilistic bounds and estimates with insight on oversampling-power iteration balance}
Evaluating the randomized SVD with a fixed budget on the number of matrix-vector multiplications, the computational resource can be leveraged in two ways -- oversampling (characterized by $l-k$) and power iterations (characterized by $q$). A natural question is \emph{how to distribute the computation between oversampling and power iterations for better subspace approximations}? 

Answers to this question are problem-dependent: when aiming to minimize the canonical angles between the true and approximated leading singular subspaces, the prior probabilistic bounds and estimates on the canonical angles provide primary insights. 
To be precise, with isotropic random subspace embeddings and sufficient oversampling, the accuracy of subspace approximations depends jointly on the spectra of the target matrices, oversampling, and the number of power iterations. 
In this work, we present a set of prior probabilistic bounds that precisely quantify the relative benefits of oversampling versus power iterations.
Specifically, the canonical angle bounds in \Cref{thm:space_agnostic_bounds}
\begin{enumerate}[label=(\roman*)]
    \item provide statistical guarantees that are asymptotically tight under sufficient oversampling (\ie, $l=\Omega\rbr{k}$),
    \item unveil a clear balance between oversampling and power iterations for random subspace approximations with given spectra,
    \item can be evaluated in $O(\rank\rbr{\Ab})$ time given access to the (true/estimated) spectra while providing valuable estimations for canonical angles in practice with moderate oversampling (\eg, $l \ge 1.6k$).
\end{enumerate} 
Further, inspired by the derivation of the prior probabilistic bounds, we propose unbiased estimates for the canonical angles with respect to given spectra that admit efficient evaluation and concentrate well empirically.

\paragraph{Posterior residual-based guarantees}
Alongside the prior probabilistic bounds, we present two sets of posterior canonical angle bounds that hold in the deterministic sense and can be approximated efficiently based on the residuals and the spectrum of $\Ab$.

\paragraph{Numerical comparisons}
With numerical comparisons among different canonical angle bounds on a variety of data matrices, we aim to explore the question on \emph{how the spectral decay and different algorithmic choices of randomized SVD affect the empirical effectiveness of different canonical angle bounds}.
Our numerical experiments suggest that, for matrices with subexponential spectral decay, the prior probabilistic bounds usually provide tighter (statistical) guarantees than the (deterministic) guarantees from the posterior residual-based bounds, especially with power iterations. By contrast, for matrices with exponential spectral decay, the posterior residual-based bounds can be as tight as the prior probabilistic bounds, especially with large oversampling. 
The code for numerical comparisons is available at \href{https://github.com/dyjdongyijun/Randomized_Subspace_Approximation}{https://github.com/dyjdongyijun/Randomized\_Subspace\_Approximation}.

\subsection{Related Work}

The randomized SVD algorithm (with power iterations) \cite{2011_martinsson_random1,halko2011finding} has been extensively analyzed as a low-rank approximation problem where the accuracy is usually measured in terms of residual norms, as well as the discrepancy between the approximated and true spectra \cite{halko2011finding,gu2014subspace,martinsson2019randomized,martinsson2020randomized}. For instance, 
\cite[Thm.~10.7, Thm.~10.8]{halko2011finding}
 show that for a given target rank $k$ (usually $k \ll \min\rbr{m,n}$ for the randomized acceleration to be useful), a small constant oversampling $l \ge 1.1 k$ is sufficient to guarantee that the residual norm of the resulting rank-$l$ approximation is close to the optimal rank-$k$ approximation (\ie, the rank-$k$ truncated SVD) error with high probability. Alternatively, \cite{gu2014subspace} investigates the accuracy of the individual approximated singular values $\wh\sigma_i$ and provides upper and lower bounds for each $\wh\sigma_i$ with respect to the true singular value $\sigma_i$.

In addition to providing accurate low-rank approximations, the randomized SVD algorithm also produces estimates of the leading left and right singular subspaces corresponding to the top singular values. When coupled with power iterations (\cite{halko2011finding} Algorithms 4.3 \& 4.4), such randomized subspace approximations are commonly known as randomized power (subspace) iterations. Their accuracy is explored in terms of canonical angles that measure differences between the unknown true subspaces and their random approximations \cite{boutsidis2015spectral,saibaba2018randomized,nakatsukasa2020sharp}. Generally, upper bounds on the canonical angles can be categorized into two types:
\begin{enumerate*}[label=(\roman*)]
    \item probabilistic bounds that establish prior statistical guarantees by exploring the concentration of the alignment between random subspace embeddings and the unknown true subspace, and
    \item residual-based bounds that can be computed a posteriori from the residual of the resulting approximation.
\end{enumerate*}

The existing prior probabilistic bounds on canonical angles~\cite{boutsidis2015spectral,saibaba2018randomized,masseyadmissible} mainly focus on the setting where the randomized SVD is evaluated without oversampling or with a small constant oversampling. Concretely, \cite{boutsidis2015spectral} derives guarantees for the canonical angles evaluated without oversampling (\ie, $l=k$) in the context of spectral clustering. Further, by taking a small constant oversampling (\eg, $l \ge k+2$) into account, Saibaba~\cite{saibaba2018randomized} provides a comprehensive analysis for an assortment of canonical angles between the true and approximated leading singular spaces. {Compared with} our results (\Cref{thm:space_agnostic_bounds}), in both no-oversampling and constant-oversampling regimes, the basic forms of the existing prior probabilistic bounds (\eg, \cite{saibaba2018randomized} Theorem 1) generally depend on the unknown singular subspace $\Vb$. Although such dependence is later lifted using the isotropicity and the concentration of the randomized subspace embedding $\Omegab$ (\eg, \cite{saibaba2018randomized} Theorem 6), the separated consideration on the spectra and the singular subspaces introduces unnecessary compromise to the upper bounds (as we will discuss in \Cref{remark:prior_bound_compare_exist}). In contrast, by allowing a more generous choice of multiplicative oversampling $l=\Omega\rbr{k}$, we present a set of space-agnostic bounds 
{(i.e., bounds that hold regardless of the singular vectors of $\Ab$)}
based on an integrated analysis of the spectra and the singular subspaces that appears to be tighter both from derivation and in practice. 

The classical Davis-Kahan $\sin\theta$ and $\tan\theta$ theorems~\cite{daviskahan} for eigenvector perturbation can be used to compute deterministic and computable bounds for the canonical angles. These bounds have the advantage that they give strict bounds (up to the estimation of the so-called gap) rather than estimates or bounds that hold with high probability (although, as we argue below, the failure probability can be taken to be negligibly low). The Davis-Kahan theorems have been extended to perturbation of singular vectors by Wedin~\cite{wedin1973perturbation}, and recent work~\cite{nakatsukasa2020sharp} derives perturbation bounds for singular vectors computed using a subspace projection method. In this work, we establish canonical angle bounds for the singular vectors in the context of (randomized) subspace iterations. Our results indicate clearly that the accuracy of the right and left singular vectors are usually not identical (i.e., $\Vb$ is more accurate with Algorithm~\ref{algo:rsvd_power_iterations}). 

As a roadmap, we formalize the problem setup in \Cref{sec:problem_setup}, including a brief review of the randomized SVD and canonical angles. In \Cref{sec:space_agnostic}, we present the prior probabilistic space-agnostic bounds. Subsequently, in \Cref{sec:space_agnostic_estimation}, we describe a set of unbiased canonical angle estimates that is closely related to the space-agnostic bounds. Then in \Cref{sec:separate_spec_bounds}, we introduce two sets of posterior residual-based bounds. Finally, in \Cref{sec:experiments}, we instantiate the insight cast by the space-agnostic bounds on the balance between oversampling and power iterations and demonstrate the empirical effectiveness of different canonical angle bounds and estimates with numerical comparisons.

\section{Problem Setup}\label{sec:problem_setup}

In this section, we first recapitulate the randomized SVD algorithm (with power iterations) \cite{halko2011finding} for which we analyze the accuracy of the resulting singular subspace approximations. Then, we review the notion of canonical angles \cite{golub2013} that quantify the {distance} between two subspaces of the same Euclidean space.

\subsection{Notation} 
For any $k \in \N$, we denote $[k]=\cbr{1,\cdots,k}$. {We adapt the MATLAB notation for matrix slicing and stacking throughout this work. In particular, for any set of matrices $\csepp{\Mb_i \in \C^{d \times d_i}}{i \in [n]}$, we denote $\sbr{\Mb_1,\cdots,\Mb_n} \in \C^{d \times \rbr{\sum_{i=1}^{n} d_i}}$ as the horizontal concatenation.}

We start by introducing notations for the SVD of a given matrix $\Ab \in \C^{m \times n}$ of rank $r$:
\begin{align*}
    \Ab 
    = 
    \underset{m \times r}{\Ub}\ \underset{r \times r}{\Sigmab}\ \underset{r \times n}{\Vb^*} 
    = 
    \bmat{\ub_{1} & \dots & \ub_{r}} \bmat{\sigma_{1} && \\ &\ddots& \\ && \sigma_{r}} \bmat{\vb_{1}^* \\ \vdots \\ \vb_{r}^*}.
\end{align*}
For any $1 \le k \le r$, we let $\Ub_{k} \triangleq \sbr{\ub_{1},\dots,\ub_{k}}$ and $\Vb_{k} \triangleq \sbr{\vb_{1},\dots,\vb_{k}}$ denote the orthonormal bases of the dimension-$k$ left and right singular subspaces of $\Ab$ corresponding to the top-$k$ singular values, while $\Ub_{r \setminus k} \triangleq\sbr{\ub_{k+1},\dots,\ub_{r}}$ and $\Vb_{r \setminus k} \triangleq\sbr{\vb_{k+1},\dots,\vb_{r}}$ are orthonormal bases of the respective orthogonal complements. The diagonal submatrices consisting of the spectrum, $\Sigmab_{k} \triangleq \diag\rbr{\sigma_{1}, \dots, \sigma_{k}}$ and $\Sigmab_{r \setminus k} \triangleq \diag\rbr{\sigma_{k+1}, \dots, \sigma_{r}}$, follow analogously.
{We use $\|\cdot \|_2$ to denote the spectral norm (largest singular value) of a matrix, and $\|\cdot \|_F$ denotes the Frobenius norm. 
Equations and inequalities involving $\nnbr{\cdot}$ hold for any (fixed) unitarily invariant norm.
}

Meanwhile, for the QR decomposition of an arbitrary matrix $\Mb \in \C^{d \times l}$ ($d \geq l$) {with full column rank $\rank(\Mb)=l$}, we denote $\Mb = \bmat{\Qb_{\Mb} & \Qb_{\Mb,\perp}} \bmat{\Rb_{\Mb} \\ \b{0}}$ such that {$\Qb_{\Mb} = \ortho\rbr{\Mb} \in \C^{d \times l}$} and $\Qb_{\Mb,\perp} \in \C^{d \times (d-l)}$ consist of orthonormal bases of the subspace spanned by the columns of $\Mb$ and its orthogonal complement.
Furthermore, we denote the {singular values} of $\Mb$ by $\sigma\rbr{\Mb}$, a $\rank(\Mb) \times \rank(\Mb)$ diagonal matrix with singular values $\sigma_1\rbr{\Mb} \ge \dots \ge \sigma_{\rank(\Mb)}\rbr{\Mb} > 0$ on the diagonal. 

Throughout, it is helpful to have in mind the ordering $k\leq l \ll r\leq \min(m,n)$, where $k$ is the dimension of the subspace of interest, $l$ is the (oversampled) dimension of subspace used in the algorithm, and $r=\mbox{rank}(\Ab)\leq \min(m,n)$. 

Given an arbitrary distribution $P: \C \to [0,1]$, we denote $\Omegab \sim P\rbr{\C^{n \times l}}$ as a random matrix with $\iid$ entries $\Omega_{ij} \sim P$ for all $i \in [n],~ j \in [l]$.

{We adapt the standard asymptotic notations as follows: for any functions $f,g: \R_+ \to \R_+$, we write $f = O\rbr{g}$ if there exists some constant $C>0$ such that $f(x) \leq C g(x)$ for all $x \in \R_+$; $f = \Omega\rbr{g}$ if $g = O\rbr{f}$; $f = \Theta\rbr{g}$ if $f = O\rbr{g}$ and $f = \Omega\rbr{g}$.}

\subsection{Randomized SVD and Power Iterations}

\begin{algorithm}
\caption{Randomized SVD (with power iterations) \cite{halko2011finding}}\label{algo:rsvd_power_iterations}
\begin{algorithmic}[1]
\REQUIRE $\Ab \in \C^{m \times n}$, power $q \in \cbr{0,1,2,\dots}$, oversampled rank $l \in \N$ ($l< r = \rank(\Ab)$)
\ENSURE $\wh\Ub_l \in \C^{m \times l}$, $\wh\Vb_l \in \C^{n \times l}$, $\wh\Sigmab_l \in \C^{l \times l}$ such that $\wh\Ab_l = \wh\Ub_l \wh\Sigmab_l \wh\Vb_l^*$

\STATE Draw $\Omegab \sim P\rbr{\C^{n \times l}}$ 
with $\Omega_{ij} \sim \Ncal\rbr{0,l^{-1}}~\iid$ such that $\E\sbr{\Omegab \Omegab^*} = \Ib_n$
\STATE $\Xb^{(q)} = \rbr{\Ab \Ab^*}^q \Ab \Omegab$
\STATE $\Qb_\Xb = \ortho\rbr{\Xb^{(q)}}$
\STATE {$\sbr{\wt\Ub_l ,\wh\Sigmab_l, \wh\Vb_l^*} = \mathrm{svd}\rbr{\Qb_\Xb^* \Ab}$} (where $\wt\Ub_l \in \C^{l \times l}$)
\STATE $\wh\Ub_l = \Qb_\Xb \wt\Ub_l$
\end{algorithmic}   
\end{algorithm} 

As described in \Cref{algo:rsvd_power_iterations}, the randomized SVD provides a rank-$l$ ($l \ll \min(m,n)$) approximation of $\Ab \in \C^{m \times n}$ while {granting} provable acceleration to the truncated SVD {computation} -- $O\rbr{mnl(2q+1)}$ {when $\Omegab$ is chosen to be a Gaussian random matrix}\footnote{Asymptotically, there exist deterministic iterative algorithms for the truncated SVD (\eg, based on Lanczos iterations (\cite{trefethen1997} Algorithm 36.1)) that run in $O\rbr{mnl}$ time. However, compared with these inherently sequential iterative algorithms, the randomized SVD can be executed much more efficiently in practice, even with power iterations (\ie, $q>0$),  since the $O\rbr{mnl(2q+1)}$ computation bottleneck in \Cref{algo:rsvd_power_iterations} involves only matrix-matrix multiplications which are easily parallelizable and highly optimized.}. Such efficiency improvement is achieved by first {projecting} the high-dimensional row (column) space of $\Ab$ to a low-dimensional subspace via a Johnson-Lindenstrauss transform\footnote{Throughout this work, we focus on Gaussian random matrices (\Cref{algo:rsvd_power_iterations}, Line 1) in the sake of theoretical guarantees, \ie, $\Omegab$ being isotropic and rotationally invariant.} (JLT) $\Omegab$ (known as ``sketching''). Then, the SVD of the resulting column (row) sketch $\Xb = \Ab\Omegab$ can be evaluated efficiently in $O\rbr{ml^2}$ time, and the rank-$l$ approximation can be constructed accordingly.

The spectral decay in $\Ab$ has a significant impact on the accuracy of the resulting low-rank approximation from \Cref{algo:rsvd_power_iterations} (as suggested in 
{\cite[Thm.~10.7, Thm.~10.8]{halko2011finding}}). To remediate the performance of \Cref{algo:rsvd_power_iterations} on matrices with flat spectra, power iterations (\Cref{algo:rsvd_power_iterations}, Line 2) {($q\geq 1$)} are usually incorporated to enhance the spectral decay. However, without proper orthogonalization, plain power iterations can be numerically unstable, especially for ill-conditioned matrices. For stable power iterations, starting with $\Xb = \Ab\Omegab \in \C^{m \times l}$ of full column rank (which holds almost surely for Gaussian random matrices), we incorporate orthogonalization in each 
power iteration 
{(or one can do so selectively~\cite[\S.~6.1]{stewart2})}
via the reduced unpivoted QR factorization (each with complexity $O(ml^2)$). Let $\ortho\rbr{\Xb} = \Qb_\Xb \in \C^{m \times l}$ be an orthonormal basis of $\Xb$ produced by the QR factorization. Then, the stable evaluation of $q$ power iterations (\Cref{algo:rsvd_power_iterations}, Line 2) can be expressed as:
\begin{align}\label{eq:stable_power_iter}
    \Xb^{(0)} \gets \ortho\rbr{\Ab\Omegab}, \quad
    \Xb^{(i)} \gets \ortho\rbr{\Ab \ortho\rbr{\Ab^* \Xb^{(i-1)}}} \ \forall\ i \in [q].
\end{align}
 
Notice that in \Cref{algo:rsvd_power_iterations}, with $\Xb = \Xb^{(q)}$, the approximated rank-$l$ SVD of $\Ab$ can be expressed as 
\begin{equation}\label{eq:whab}
\wh\Ab_l = \wh\Ub_l \wh\Sigmab_l \wh\Vb_l^* = \Xb \Yb^*\quad  \mbox{where}\quad \Yb = \Xb^\dagger \Ab.  
\end{equation}
With $\wh\Ub_l$ and $\wh\Vb_l$ characterizing the approximated $l$-dimensional left and right leading singular subspaces, let $\wh\Ub_{m \setminus l} \in \C^{m \times (m-l)}$ and $\wh\Vb_{n \setminus l} \in \C^{n \times (n-l)}$ denote an arbitrary pair of their respective orthogonal complements. For any $1 \le k < l$, we further denote the partitions $\wh\Ub_l = \sbr{\wh\Ub_k, \wh\Ub_{l \setminus k}}$ and $\wh\Vb_l = \sbr{\wh\Vb_k, \wh\Vb_{l \setminus k}}$ where $\wh\Ub_k \in \C^{m \times k}$ and $\wh\Vb_k \in \C^{n \times k}$, respectively.

\subsection{Canonical Angles}

Now, we review the notion of canonical angles \cite{golub2013} that measure distances between two subspaces $\Ucal$, $\Vcal$ of an arbitrary Euclidean space $\C^d$.

\begin{definition}[Canonical angles, \cite{golub2013}]\label{def:canonical-angles}
Given two subspaces $\Ucal,\Vcal \subseteq \C^d$ with dimensions $\dim\rbr{\Ucal}=l$ and $\dim\rbr{\Vcal}=k$ (assuming $l \geq k$ without loss of generality), the canonical angles, denoted by $\angle\rbr{\Ucal,\Vcal}=\diag\rbr{\theta_1,\dots,\theta_k}$, consist of $k$ angles that measure the alignment between $\Ucal$ and $\Vcal$, defined recursively such that
\begin{align*}
    \ub_i, \vb_i ~\triangleq~
    &\argmax~\ub_i^*\vb_i \\
    \t{s.t.}~
    &\ub_i \in \rbr{\Ucal \setminus \spn\cbr{\ub_{\iota}}_{\iota=1}^{i-1}} \cap \SSS^{d-1},\\ 
    &\vb_i \in \rbr{\Vcal \setminus \spn\cbr{\vb_{\iota}}_{\iota=1}^{i-1}} \cap \SSS^{d-1}\\
    \cos \theta_i ~=~
    &\ub_i^* \vb_i \quad \forall~ i=1,\dots,k, \quad
    0 \leq \theta_1 \leq \dots \leq \theta_k \leq \pi/2.
\end{align*}
\end{definition}

For arbitrary full-rank matrices 
{$\Mb \in \C^{d \times l}$ and $\Nb \in \C^{d \times k}$}
(assuming $k \leq l \leq d$ without loss of generality), let $\angle\rbr{\Mb,\Nb} \triangleq \angle\rbr{\spn(\Mb),\spn(\Nb)}$ denote the canonical angles between the corresponding spanning subspaces in $\C^d$. For each $i \in [k]$, let $\angle_i\rbr{\Mb,\Nb}$ be the $i$-th (smallest) canonical angle {defined as} $\cos\angle_i\rbr{\Mb,\Nb} = \sigma_i\rbr{\Qb_{\Mb}^*\Qb_{\Nb}}$ {such that} $\sin\angle_i\rbr{\Mb,\Nb} = \sigma_{k-i+1}\rbr{\rbr{\Ib-\Qb_{\Mb}\Qb_{\Mb}^*}\Qb_{\Nb}}$ (\cf \cite{bjorck1973numerical} Section 3).

With the unknown true rank-$k$ truncated SVD $\Ab_k = \Ub_k \Sigmab_k \Vb_k^*$ and an approximated rank-$l$ SVD $\wh\Ab_l = \wh\Ub_l \wh\Sigmab_l \wh\Vb_l^*$ from \Cref{algo:rsvd_power_iterations}, in this work, we mainly focus on prior and posterior guarantees for the canonical angles $\angle\rbr{\Ub_k,\wh\Ub_l}$ and $\angle\rbr{\Vb_k,\wh\Vb_l}$. Meanwhile, in \Cref{thm:separate_gap_bounds}, we present a set of posterior residual-based upper bounds for the canonical angles $\angle\rbr{\Ub_k,\wh\Ub_k}$ and $\angle\rbr{\Vb_k,\wh\Vb_k}$ as corollaries.

\section{Space-agnostic Bounds under Sufficient Oversampling}\label{sec:space_agnostic}

We start by pointing out the intuition that, under sufficient oversampling, with Gaussian random matrices whose distribution is orthogonally invariant, the alignment between the approximated and true subspaces are independent of the unknown true subspaces, \ie, the canonical angles are space-agnostic, as reflected in the following theorem.

\begin{theorem}\label{thm:space_agnostic_bounds}
For a rank-$l$ randomized SVD (\Cref{algo:rsvd_power_iterations}) with {a} Gaussian embedding {$\Omegab$} and $q \ge 0$ power iterations, when the oversampled rank $l$ satisfies $l = \Omega\rbr{k}$ (where $k$ is the target rank, $k < l < r = \rank(\Ab)$) and $q$ is reasonably small such that $\eta \dfeq \frac{\rbr{\sum_{j=k+1}^r \sigma_j^{4q+2}}^2}{\sum_{j=k+1}^r \sigma_j^{2(4q+2)}}$
\footnote{Notice that $1 < \eta \le r-k$. To the extremes, $\eta = r-k$ when the tail is flat $\sigma_{k+1}=\dots=\sigma_r$; while $\eta \to 1$ when $\sigma_{k+1} \gg \sigma_j$ for all $j=k+2,\dots,r$. In particular, with a relatively flat tail $\Sigmab_{r \setminus k}$ and a reasonably small $q$ (recall that $q=1,2$ is usually sufficient in practice), we have $\eta = \Theta\rbr{r-k}$, and the assumption can be simplified as $r-k=\Omega\rbr{l}$. Although exponential tail decay can lead to small $\eta$ and may render the assumption infeasible in theory, in practice, simply taking $r-k=\Omega\rbr{l}$, $l=\Omega\rbr{k}$, $\eps_1 = \sqrt{\frac{k}{l}}$ and $\eps_2=\sqrt{\frac{l}{r-k}}$ is sufficient to ensure the validity of upper bounds when $q \le 10$ even for matrices with rapid tail decay, as shown in \Cref{subsec:experiment_canonical_angle}.}
satisfies $\eta = \Omega\rbr{l}$, with high probability (at least $1-e^{-\Theta(k)}-e^{-\Theta(l)}$), there exist distortion factors $0< \eps_1, \eps_2 < 1$ such that
\begin{align}
    \label{eq:space_agnostic_left}
    &\sin\angle_i\rbr{\Ub_k, \wh\Ub_l} \le 
    \rbr{1+ \frac{1-\eps_1}{1+\eps_2} \cdot \frac{l}{\sum_{j=k+1}^r \sigma_j^{4q+2}} \cdot \sigma_i^{4q+2}}^{-\frac{1}{2}}
    \\
    \label{eq:space_agnostic_right}
    &\sin\angle_i\rbr{\Vb_k, \wh\Vb_l} \le 
    \rbr{1+ \frac{1-\eps_1}{1+\eps_2} \cdot \frac{l}{\sum_{j=k+1}^r \sigma_j^{4q+4}} \cdot \sigma_i^{4q+4}}^{-\frac{1}{2}}
\end{align}
for all $i \in [k]$, where $\eps_1 = \Theta\rbr{\sqrt{\frac{k}{l}}}$ and $\epsilon_2 = \Theta\rbr{\sqrt{\frac{l}{\eta}}}$.
Furthermore, both bounds are asymptotically tight:
\begin{align}
    \label{eq:space_agnostic_lower_left}
    &\sin\angle_i\rbr{\Ub_k, \wh\Ub_l} \ge \rbr{1+O\rbr{\frac{l \cdot \sigma_i^{4q+2}}{\sum_{j=k+1}^r \sigma_j^{4q+2}} }}^{-\frac{1}{2}}
    \\
    \label{eq:space_agnostic_lower_right}
    &\sin\angle_i\rbr{\Vb_k, \wh\Vb_l} \ge \rbr{1+O\rbr{\frac{l \cdot \sigma_i^{4q+4}}{\sum_{j=k+1}^r \sigma_j^{4q+4}} }}^{-\frac{1}{2}}
\end{align}
where $O\rbr{\cdot}$ suppresses the distortion factors $\frac{1+\eps_1}{1-\eps_2}$
\footnote{
    Despite the asymptotic tightness of \Cref{thm:space_agnostic_bounds} theoretically, in practice, we observe that the empirical validity of lower bounds is more restrictive on oversampling than that of upper bounds. In specific, the numerical observations in \Cref{subsec:experiment_canonical_angle} suggest that $l \ge 1.6 k$ is usually sufficient for the upper bounds to hold; whereas the empirical validity of lower bounds generally requires more aggressive oversampling of at least $l \ge 4 k$, also with slightly larger constants associated with $\eps_1$ and $\eps_2$, as demonstrated in \Cref{subapx:sup_exp_upper_lower_space_agnostic}.
}.
\end{theorem}   

The main insights provided by \Cref{thm:space_agnostic_bounds} include:
\begin{enumerate*}[label=(\roman*)]
    \item improved statistical guarantees for canonical angles under sufficient oversampling (\ie, $l=\Omega\rbr{k}$), as discussed later in \Cref{remark:prior_bound_compare_exist},
    \item a clear view of the balance between oversampling and power iterations for random subspace approximations with given spectra, as instantiated in \Cref{subsec:balance_oversampling_power_iterations_example}, and
    \item affordable upper bounds that can be evaluated in $O(\rank\rbr{\Ab})$ time with access to the (true/estimated) spectra and hold in practice with only moderate oversampling (\eg, $l \ge 1.6k$), as shown in \Cref{subsec:experiment_canonical_angle}.
\end{enumerate*} 
{We also note that when the true singular values in \Cref{thm:space_agnostic_bounds} are unknown, they may be replaced by the approximated singular values $\wh\Sigmab_l$ from \Cref{algo:rsvd_power_iterations} in practice, \cf~\Cref{subsec:experiment_canonical_angle}.}

\begin{proof}[Proof of \Cref{thm:space_agnostic_bounds}]
We show the derivation of \Cref{eq:space_agnostic_left} for left canonical angles $\sin\angle_i\rbr{\Ub_k, \wh\Ub_l}$.
The derivation for $\sin\angle_i\rbr{\Vb_k, \wh\Vb_l}$ in \Cref{eq:space_agnostic_right} follows {directly} by replacing the exponent $4q+2$ in \Cref{eq:space_agnostic_left} with $4q+4$ in \Cref{eq:space_agnostic_right}. This slightly larger exponent comes from the additional half power iteration associated with $\wh\Vb_l$ in \Cref{algo:rsvd_power_iterations} (as discussed in \Cref{remark:compare_left_right}), an observation made also in~\cite{saibaba2018randomized}. 

For rank-$l$ randomized SVD with Gaussian embedding $\Omegab \in \C^{n \times l}$ and $q$ power iterations, we denote $\Omegab_1 \triangleq \Vb_{k}^* \Omegab$ and $\Omegab_2 \triangleq \Vb_{r \setminus k}^* \Omegab$, as well as their weighted correspondences $\wt\Omegab_1 \dfeq \Sigmab_k^{2q+1} \Omegab_1$ and $\wt\Omegab_2 \dfeq \Sigmab_{r \setminus k}^{2q+1} \Omegab_2$, such that
\begin{align*}
    \wt\Xb \triangleq \Ub^* \Xb = \Ub^* \rbr{\Ab \Ab^*}^q \Ab \Omegab = \bmat{\Sigmab_k^{2q+1} \Omegab_1 \\ \Sigmab_{r \setminus k}^{2q+1} \Omegab_2} = \bmat{\wt\Omegab_1 \\ \wt\Omegab_2}.
\end{align*}
Since $\spn\rbr{\Ub_k} \subseteq \spn\rbr{\Ub}$ and $\spn\rbr{\Xb} \subseteq \spn\rbr{\Ub}$, we have $\Ub \Ub^* \Ub_k = \Ub_k$ and $\Ub \Ub^* \Xb = \Xb$, respectively. Then for all $i \in [k]$, 
\begin{align*}
    \sin\angle_{k-i+1}\rbr{\Ub_k,\wh\Ub_l} 
    = &\sigma_i\rbr{\rbr{\Ib_m - \wh\Ub_l \wh\Ub_l^*} \Ub_k}
    \\
    \rbr{\Ub \Ub^* \Ub_k = \Ub_k}~
    = &\sigma_i\rbr{\rbr{\Ub \Ub^* - \wh\Ub_l \wh\Ub_l^*} \Ub_k}
    \\
    \rbr{\mbox{span}(\Ub_l)=\mbox{span}(\Xb)\  \mbox{by~\eqref{eq:whab}}}~
    = &\sigma_i\rbr{\rbr{\Ub \Ub^* - \Xb \Xb^{\pinv}} \Ub_k} 
    \\
    \rbr{\Ub \Ub^* \Xb = \Xb}~
    = &\sigma_i\rbr{\Ub\Ub^* \rbr{\Ib_m - \Xb \Xb^{\pinv}} \Ub \Ub^* \Ub_k} 
    \\
    \rbr{\t{$\Ub$ consists of orthonormal columns}}~
    = &\sigma_i\rbr{\Ub^* \rbr{\Ib_m - \Xb \Xb^{\pinv}} \Ub \Ub^* \Ub_k} 
    \\
    \rbr{\Ub^*\Ub=\Ib_r,~ \Xb^\pinv\Ub=\rbr{\Ub^*\Xb}^\pinv {=\wt\Xb^\pinv}}~
    = &\sigma_i\rbr{\rbr{\Ib_r - \wt\Xb \wt\Xb^{\pinv}} \bmat{\Ib_k \\ \b{0}} }.
\end{align*}
Since $\Xb$ is assumed to {have} full column rank, we have $\wt\Xb \wt\Xb^{\pinv} = \wt\Xb \rbr{\wt\Xb^* \wt\Xb}^{-1} \wt\Xb^*$ {(which is an orthogonal projection)}, and
\begin{align*}
    &\bmat{\Ib_k & \b0} \rbr{\Ib_r - \wt\Xb \wt\Xb^{\pinv}} \bmat{\Ib_k \\ \b0} 
    \\
    = &\bmat{\Ib_k & \b0} \rbr{\Ib_r - \bmat{\wt\Omegab_1 \\ \wt\Omegab_2} \rbr{\wt\Omegab_1^* \wt\Omegab_1 + \wt\Omegab_2^* \wt\Omegab_2}^{-1} \bmat{\wt\Omegab_1^* & \wt\Omegab_2^*}} \bmat{\Ib_k \\ \b0}
    \\
    = &\Ib_k - \wt\Omegab_1 \rbr{\wt\Omegab_1^* \wt\Omegab_1 + \wt\Omegab_2^* \wt\Omegab_2}^{-1} \wt\Omegab_1^*
    \\
    \rbr{\t{Woodbury identity}}~ = &\rbr{\Ib_k + \wt\Omegab_1 \rbr{\wt\Omegab_2^* \wt\Omegab_2}^{-1} {\wt\Omegab_1^*}}^{-1}.
\end{align*}    
Therefore for all $i \in [k]$,
\begin{align}\label{eq:pf_sin_Uk_whUl}
\begin{split}
    \sin^2\angle_{k-i+1}\rbr{\Ub_k,\wh\Ub_l}  
    = &\sigma_i \rbr{\bmat{\Ib_k & \b0} \rbr{\Ib_r - \wt\Xb \wt\Xb^{\pinv}} \bmat{\Ib_k \\ \b0}} \\
    = &\sigma_i\rbr{\rbr{\Ib_k + \wt\Omegab_1 \rbr{\wt\Omegab_2^* \wt\Omegab_2}^{-1} \wt\Omegab_1^*}^{-1}}.
\end{split}
\end{align}

{By the orthogonal invariance of the distribution of Gaussian embeddings $\Omegab$ together with the orthonormality of $\Vb_k \perp \Vb_{r \setminus k}$, we see that}
$\Omegab_1 \sim P\rbr{\C^{k \times l}}$ and $\Omegab_2 \sim P\rbr{\C^{(r-k) \times l}}$ are independent Gaussian random matrices with the same entry-wise {$\iid$} distribution $\Ncal\rbr{0,l^{-1}}$ as $\Omegab$. 
Therefore by \Cref{lemma:sample_to_population_covariance}, when $l= \Omega(k)$, with high probability (at least $1-e^{-\Theta(k)}$),
\begin{align*}
    \rbr{1-\eps_1} \Sigmab_{k}^{4q+2} 
    \aleq 
    \wt\Omegab_1 \wt\Omegab_1^*
    \aleq
    \rbr{1+\eps_1} \Sigmab_{k}^{4q+2} 
\end{align*}
for some $\epsilon_1 = \Theta\rbr{\sqrt{\frac{k}{l}}}$. 

Analogously, when $r-k \ge \eta = \frac{\tr\rbr{\Sigmab_{r \setminus k}^{4q+2}}^2}{\tr\rbr{\Sigmab_{r \setminus k}^{2(4q+2)}}} = \Omega(l)$, with high probability (at least $1-e^{-\Theta(l)}$),
\begin{align*}
    \frac{1-\eps_2}{l}\tr\rbr{\Sigmab_{r \setminus k}^{4q+2}} \Ib_l 
    \aleq 
    \wt\Omegab_2^* \wt\Omegab_2
    \aleq
    \frac{1+\eps_2}{l}\tr\rbr{\Sigmab_{r \setminus k}^{4q+2}} \Ib_l
\end{align*}
for some $\epsilon_2 = \Theta\rbr{\sqrt{\frac{l}{\eta}}}$.

Therefore by {the} union bound, we have with high probability (at least $1-e^{-\Theta(k)}-e^{-\Theta(l)}$) that,
\begin{align*}
    \rbr{\Ib_k + \wt\Omegab_1 \rbr{\wt\Omegab_2^* \wt\Omegab_2}^{-1} \wt\Omegab_1^*}^{-1}
    \aleq
    \rbr{\Ib_k + \frac{1-\eps_1}{1+\eps_2} \cdot \frac{l}{\tr\rbr{\Sigmab_{r \setminus k}^{4q+2}}} \cdot \Sigmab_k^{4q+2}}^{-1},
\end{align*}    
which leads to \Cref{eq:space_agnostic_left}, while the tightness is implied by 
\begin{align*}
    \rbr{\Ib_k + \wt\Omegab_1 \rbr{\wt\Omegab_2^* \wt\Omegab_2}^{-1} \wt\Omegab_1^*}^{-1}
    \ageq
    \rbr{\Ib_k + \frac{1+\eps_1}{1-\eps_2} \cdot \frac{l}{\tr\rbr{\Sigmab_{r \setminus k}^{4q+2}}} \cdot \Sigmab_k^{4q+2}}^{-1}.
\end{align*}   

The proof of \Cref{eq:space_agnostic_right} follows analogously by replacing the exponents $2q+1$ and $4q+2$ with $2q+2$ and $4q+4$, respectively.
\end{proof} 

\begin{remark}[Comparison with existing probabilistic bounds]\label{remark:prior_bound_compare_exist}
    With access to the unknown right singular subspace $\Vb$, let $\Omegab_1 \triangleq \Vb_{k}^* \Omegab$ and $\Omegab_2 \triangleq \Vb_{r \setminus k}^* \Omegab$. Then, {Saibaba~\cite[Thm.~1]{saibaba2018randomized}}indicates that, for all $i \in [k]$,
    \begin{align}
        \label{eq:saibaba2018_thm1_left}
        &\sin\angle_i\rbr{\Ub_k, \wh\Ub_l} \le
        \rbr{1 + \frac{\sigma_i^{4q+2}}{\sigma_{k+1}^{4q+2} \nbr{\Omegab_2 \Omegab_1^\dagger}_2^2}}^{-\frac{1}{2}},
        \\
        \label{eq:saibaba2018_thm1_right}
        &\sin\angle_i\rbr{\Vb_k, \wh\Vb_l} \le 
        \rbr{1 + \frac{\sigma_i^{4q+4}}{\sigma_{k+1}^{4q+4} \nbr{\Omegab_2 \Omegab_1^\dagger}_2^2}}^{-\frac{1}{2}}.
    \end{align} 
    Further, leveraging existing results on concentration properties of the independent and isotropic Gaussian random matrices $\Omegab_1$ and $\Omegab_2$ (\eg, from the proof of \cite{halko2011finding} Theorem 10.8), \cite{saibaba2018randomized} shows that, when $l \ge k+2$, for any $\delta \in (0,1)$, with probability at least $1-\delta$,
    \begin{align*}
        \nbr{\Omegab_2 \Omegab_1^\dagger}_2 \le 
        \frac{e \sqrt{l}}{l-k+1} \rbr{\frac{2}{\delta}}^{\frac{1}{l-k+1}} \rbr{\sqrt{n-k}+\sqrt{l}+\sqrt{2 \log\frac{2}{\delta}}}.
    \end{align*}    
    
    Without loss of generality, we consider the bounds on $\sin\angle\rbr{\Ub_k,\wh\Ub_l}$. Comparing to the existing bound in \Cref{eq:saibaba2018_thm1_left}, under multiplicative oversampling ($l=\Omega(k)$, $r=\Omega(l)$), \Cref{eq:space_agnostic_left} in \Cref{thm:space_agnostic_bounds} captures the spectral decay on the tail by replacing the denominator term 
    \begin{align*}
        \sigma_{k+1}^{4q+2} \nbr{\Omegab_2 \Omegab_1^\dagger}_2^2
        \quad \text{with} \quad
        \frac{1+\eps_2}{1-\eps_1} \cdot \frac{1}{l}\sum_{j=k+1}^r \sigma_j^{4q+2} = \Theta\rbr{\frac{1}{l}\sum_{j=k+1}^r \sigma_j^{4q+2}}.
    \end{align*}
    We observe that $\frac{1}{l}\sum_{j=k+1}^r \sigma_j^{4q+2} \le \frac{r-k}{l} \sigma_{k+1}^{4q+2}$; while \Cref{lemma:sample_to_population_covariance} implies that, for independent Gaussian random matrices $\Omegab_1 \sim P\rbr{\C^{k \times l}}$ and $\Omegab_2 \sim P\rbr{\C^{(r-k) \times l}}$ with \iid entries from $\Ncal\rbr{0,l^{-1}}$,
    \begin{align*}
        \E\sbr{\nbr{\Omegab_2 \Omegab_1^\dagger}_2^2} = \E_{\Omegab_1}\sbr{\nbr{\rbr{\Omegab_1^\pinv}^* \E_{\Omegab_2}\sbr{\Omegab_2^* \Omegab_2} \Omegab_1^\dagger}_2} = \frac{r-k}{l} \E_{\Omegab_1}\sbr{\nbr{\rbr{\Omegab_1 \Omegab_1^*}^\dagger}_2}.
    \end{align*}
    With non-negligible spectral decay on {the} tail such that $\frac{1}{l}\sum_{j=k+1}^r \sigma_j^{4q+2} \ll \frac{r-k}{l} \sigma_{k+1}^{4q+2}$, when $\frac{1+\eps_2}{1-\eps_1} \cdot \frac{1}{l}\sum_{j=k+1}^r \sigma_j^{4q+2} \ll \sigma_{k+1}^{4q+2} \nbr{\Omegab_2 \Omegab_1^\dagger}_2^2$, \Cref{eq:space_agnostic_left} provides {a} tighter statistical guarantee than \Cref{eq:saibaba2018_thm1_left}, which is also confirmed by numerical observations in \Cref{subsec:experiment_canonical_angle}.

    From the derivation perspective, such improvement is achieved by taking an integrated view on the concentration of $\Sigmab_{r \setminus k}^{2q+1}\Omegab_2$ (commonly used in analyzing randomized low-rank approximation error~\cite{halko2011finding,gittens2013revisiting,frangella2023randomized,tropp2023randomized}), instead of considering the spectrum and the unknown singular subspace separately.
\end{remark}

\section{Unbiased Space-agnostic Estimates}\label{sec:space_agnostic_estimation}

A natural corollary from the proof of \Cref{thm:separate_gap_bounds} is unbiased estimates for the canonical angles that hold for arbitrary oversampling (\ie, for all $l \ge k$). Further, we will subsequently {see} in \Cref{sec:experiments} that such unbiased estimates also enjoy good empirical concentration.

\begin{proposition}\label{prop:space_agnostic_estimation}
For a rank-$l$ randomized SVD (\Cref{algo:rsvd_power_iterations}) with the Gaussian embedding $\Omegab \sim P\rbr{\C^{n \times l}}$ such that $\Omega_{ij} \sim \Ncal\rbr{0,l^{-1}}~\iid$ and $q \ge 0$ power iterations, for all $i \in [k]$,
\begin{align}\label{eq:space_agnostic_estimation_left}
    \E_{\Omegab} \sbr{\sin\angle_i\rbr{\Ub_k,\wh\Ub_l}}
    = &\E_{\Omegab_1', \Omegab_2'} \sbr{\sigma_i^{-\frac{1}{2}} \rbr{\Ib_k + \Sigmab_k^{2q+1} \Omegab'_1 \rbr{\Omegab_2^{'*} \Sigmab_{r \setminus k}^{4q+2} \Omegab'_2}^{-1} \Omegab_1^{'*} \Sigmab_k^{2q+1}}},
\end{align}
and analogously,
\begin{align}\label{eq:space_agnostic_estimation_right}
    \E_{\Omegab} \sbr{\sin\angle_i\rbr{\Vb_k,\wh\Vb_l}}
    = &\E_{\Omegab_1', \Omegab_2'} \sbr{\sigma_i^{-\frac{1}{2}} \rbr{\Ib_k + \Sigmab_k^{2q+2} \Omegab'_1 \rbr{\Omegab_2^{'*} \Sigmab_{r \setminus k}^{4q+4} \Omegab'_2}^{-1} \Omegab_1^{'*} \Sigmab_k^{2q+2}}},
\end{align}
where $\Omegab'_1 \sim P\rbr{\C^{k \times l}}$ and $\Omegab'_2 \sim P\rbr{\C^{(r-k) \times l}}$ are independent Gaussian random matrices with $\iid$ entries drawn from $\Ncal\rbr{0,l^{-1}}$.
\end{proposition}

To calculate the unbiased estimate, 
{for a modest integer $N$}
we draw a set of independent Gaussian random matrices $\csepp{\Omegab^{(j)}_1 \sim P\rbr{\C^{k \times l}}, \Omegab^{(j)}_2 \sim P\rbr{\C^{(r-k) \times l}}}{j \in [N]}$ and evaluate
\begin{align*}
    &\sin\angle_i\rbr{\Ub_k,\wh\Ub_l} \approx \alpha_i = \frac{1}{N} \sum_{j=1}^N \rbr{1 + \sigma_i^2\rbr{\Sigmab_k^{2q+1} \Omegab^{(j)}_1 \rbr{\Sigmab_{r \setminus k}^{2q+1} \Omegab^{(j)}_2}^\dagger}}^{-\frac{1}{2}},
    \\
    &\sin\angle_i\rbr{\Vb_k,\wh\Vb_l} \approx \beta_i = \frac{1}{N} \sum_{j=1}^N \rbr{1 + \sigma_i^2\rbr{\Sigmab_k^{2q+2} \Omegab^{(j)}_1 \rbr{\Sigmab_{r \setminus k}^{2q+2} \Omegab^{(j)}_2}^\dagger}}^{-\frac{1}{2}},
\end{align*}
for all $i \in [k]$, which can be conducted efficiently in $O\rbr{Nrl^2}$ time. \Cref{algo:unbiased_canonical_angle_estimators} demonstrates the construction of unbiased estimates for $\E \sbr{\sin\angle_i\rbr{\Ub_k,\wh\Ub_l}}=\alpha_i$; while the unbiased estimates for $\E \sbr{\sin\angle_i\rbr{\Vb_k,\wh\Vb_l}}=\beta_i$ can be evaluated analogously by replacing Line 4 with $\wt\Omegab^{(j)}_1 = \Sigmab_k^{2q+2} \Omegab^{(j)}(1:k,:)$, $\wt\Omegab^{(j)}_2 = \Sigmab_{r \setminus k}^{2q+2} \Omegab^{(j)}(k+1:r,:)$. 

\begin{algorithm}
\caption{Unbiased canonical angle estimates}
\label{algo:unbiased_canonical_angle_estimators}
\begin{algorithmic}[1]
\REQUIRE {(Exact or estimated) singular values} $\Sigmab$, rank $k$, sample size $l \ge k$, number of power iterations $q$, number of trials $N$
\ENSURE Unbiased estimates $\E \sbr{\sin\angle_i\rbr{\Ub_k,\wh\Ub_l}}=\alpha_i$ for all $i \in [k]$

\STATE Partition $\Sigmab$ into $\Sigmab_k = \Sigmab(1:k,1:k)$ and $\Sigmab_{r \setminus k} = \Sigmab(k+1:r,k+1:r)$
\FOR{$j=1,\dots,N$}
    \STATE Draw $\Omegab^{(j)} \sim P\rbr{\C^{r \times l}}$ such that $\Omega^{(j)}_{ij} \sim \Ncal\rbr{0,l^{-1}}~\iid$ 
    \STATE $\wt\Omegab^{(j)}_1 = \Sigmab_k^{2q+1} \Omegab^{(j)}(1:k,:)$, $\wt\Omegab^{(j)}_2 = \Sigmab_{r \setminus k}^{2q+1} \Omegab^{(j)}(k+1:r,:)$
    \STATE $\sbr{\Ub_{\wt\Omegab^{(j)}_2}, \Sigmab_{\wt\Omegab^{(j)}_2}, \Vb_{\wt\Omegab^{(j)}_2}} = \t{svd}\rbr{\wt\Omegab^{(j)}_2, \t{``econ''}}$
    \STATE $\nub^{(j)} = \t{svd}\rbr{\wt\Omegab^{(j)}_1 \Vb_{\wt\Omegab^{(j)}_2} \Sigmab_{\wt\Omegab^{(j)}_2}^{-1} \Ub_{\wt\Omegab^{(j)}_2}^*}$
    \STATE $\theta^{(j)}_i = 1/\sqrt{1 + \rbr{\nu^{(j)}_i}^2}$ for all $i \in [k]$
\ENDFOR
\STATE $\alpha_i = \frac{1}{N} \sum_{j=1}^N \theta_i^{(j)}$ for all $i \in [k]$
\end{algorithmic}       
\end{algorithm} 

Experiments in \Cref{sec:experiments} show that the unbiased estimates concentrate well in practice: a sample size as small as $N=3$ is {seen to be} sufficient to provide good estimates. Further, with independent Gaussian random matrices, the unbiased estimates in \Cref{prop:space_agnostic_estimation} are space-agnostic, \ie, \Cref{eq:space_agnostic_estimation_left} and \Cref{eq:space_agnostic_estimation_right} only depend on the spectrum $\Sigmab$ but not on the unknown true singular subspaces $\Ub$ and $\Vb$. As \Cref{thm:space_agnostic_bounds}, the true singular values may be replaced by their approximations from \Cref{algo:rsvd_power_iterations} (\cf~\Cref{subsec:experiment_canonical_angle}).

\begin{proof}[Proof of \Cref{prop:space_agnostic_estimation}]
To show \Cref{eq:space_agnostic_estimation_left}, we recall from the proof of \Cref{thm:space_agnostic_bounds} that, for the rank-$l$ randomized SVD with a Gaussian embedding $\Omegab \sim P\rbr{\C^{n \times l}}$ and $q$ power iterations, $\Omegab_1 \dfeq \Vb_{k}^* \Omegab$ and $\Omegab_2 \dfeq \Vb_{r \setminus k}^* \Omegab$ are independent Gaussian random matrices with the same entry-wise distribution as $\Omegab$. Therefore, with $\rank\rbr{\Ab}=r$, for all $i \in [k]$,
\begin{align*}
    &\E_{\Omegab} \sbr{\sin\angle_{i}\rbr{\Ub_k,\wh\Ub_l}} 
    \quad \t{(Recall \Cref{eq:pf_sin_Uk_whUl})}
    \\
    = &\E_{\Omegab} \sbr{\sigma_{k-i+1}^{\frac{1}{2}}\rbr{\rbr{\Ib_k + \Sigmab_k^{2q+1} \Omegab_1 \rbr{\Omegab_2^* \Sigmab_{r \setminus k}^{4q+2} \Omegab_2}^{-1} \Omegab_1^* \Sigmab_k^{2q+1}}^{-1}}} 
    \\
    = &\E_{\Omegab} \sbr{\sigma_i^{-\frac{1}{2}}\rbr{\Ib_k + \Sigmab_k^{2q+1} \Omegab_1 \rbr{\Omegab_2^* \Sigmab_{r \setminus k}^{4q+2} \Omegab_2}^{-1} \Omegab_1^* \Sigmab_k^{2q+1}}}
    \\
    = &\underset{\Omegab_1', \Omegab_2'}{\E} \sbr{\sigma_i^{-\frac{1}{2}}\rbr{\Ib_k + \Sigmab_k^{2q+1} \Omegab'_1 \rbr{\Omegab_2^{'*} \Sigmab_{r \setminus k}^{4q+2} \Omegab'_2}^{-1} \Omegab_1^{'*} \Sigmab_k^{2q+1}}}.
\end{align*}
The unbiased estimate in \Cref{eq:space_agnostic_estimation_right} follows analogously.
\end{proof}

As a side note, we point out that, compared with the probabilistic upper bounds \Cref{eq:space_agnostic_left} and \Cref{eq:space_agnostic_right}, the estimates \Cref{eq:space_agnostic_estimation_left} and \Cref{eq:space_agnostic_estimation_right} circumvent overestimation from the operator-convexity of inversion $\sigma \to \sigma^{-1}$,
\begin{align*}
    \E_{\Omegab_2'} \sbr{\rbr{\Omegab_2^{'*} \Sigmab_{r \setminus k}^{4q+2} \Omegab'_2}^{-1}} \ageq 
    \rbr{\E_{\Omegab_2'} \sbr{\Omegab_2^{'*} \Sigmab_{r \setminus k}^{4q+2} \Omegab'_2}}^{-1},
\end{align*}
which implies that
\begin{align*}
    &\E_{\Omegab_1', \Omegab_2'} \sbr{ \rbr{\Ib_k + \Sigmab_k^{2q+1} \Omegab'_1 \rbr{\Omegab_2^{'*} \Sigmab_{r \setminus k}^{4q+2} \Omegab'_2}^{-1} \Omegab_1^{'*} \Sigmab_k^{2q+1}}^{-1} }
    \\
    &\ageq \E_{\Omegab_1'} \sbr{ \rbr{\Ib_k + \Sigmab_k^{2q+1} \Omegab'_1 \rbr{\E_{\Omegab_2'} \sbr{\Omegab_2^{'*} \Sigmab_{r \setminus k}^{4q+2} \Omegab'_2}}^{-1} \Omegab_1^{'*} \Sigmab_k^{2q+1}}^{-1} }.
\end{align*}

\section{Posterior Residual-based Bounds}\label{sec:separate_spec_bounds}

In addition to the prior probabilistic bounds and unbiased estimates, in this section, we introduce two sets of posterior guarantees for the canonical angles that hold deterministically and can be evaluated/approximated efficiently based on the residual of the resulting low-rank approximation $\wh\Ab_l = \wh\Ub_l \wh\Sigmab_l \wh\Vb_l^*$ from \Cref{algo:rsvd_power_iterations}.

\begin{remark}[Generality of residual-based bounds]\label{remark:residula_based_generality}
    It is worth highlighting that both the statements and the proofs of the posterior guarantees \Cref{thm:with_oversmp_computable_det,thm:separate_gap_bounds} to be presented are algorithm-independent. 
    In contrast to \Cref{thm:space_agnostic_bounds} and \Cref{prop:space_agnostic_estimation} whose derivation depends explicitly on the algorithm (\eg, assuming $\Omegab$ being Gaussian in \Cref{algo:rsvd_power_iterations}), the residual-based bounds in \Cref{thm:with_oversmp_computable_det,thm:separate_gap_bounds} hold for general {approximated low-rank SVDs,} $\Ab \approx \wh\Ub_l \wh\Sigmab_l\wh\Vb_l^T$.
\end{remark}   

We start with the following proposition that establishes relations between the canonical angles, the residuals, and the true spectrum $\sigma\rbr{\Ab}$.

\begin{theorem}\label{thm:with_oversmp_computable_det} 
Given any $\wh\Ub_l \in \C^{m \times l}$ and $\wh\Vb_l \in \C^{n \times l}$ with orthonormal columns such that $\range\rbr{\wh\Ub_l} \subseteq \col\rbr{\Ab}$ and $\range\rbr{\wh\Vb_l} \subseteq \row\rbr{\Ab}$, we have for each $i=1,\dots,k$ ($k \le l$),
\begin{align}\label{eq:residual_based_posterior_left}
    &\sin\angle_i\rbr{\Ub_k, \wh\Ub_l} 
    \leq
    \min \cbr{ 
    \frac{\sigma_{k-i+1}\rbr{\rbr{\Ib_m-\wh\Ub_l \wh\Ub_l^*}\Ab}}{\sigma_k},
    \frac{\sigma_1\rbr{\rbr{\Ib_m-\wh\Ub_l \wh\Ub_l^*}\Ab}}{\sigma_i} 
    },
\end{align}
while
\begin{align}\label{eq:residual_based_posterior_right}
    &\sin\angle_i\rbr{\Vb_k, \wh\Vb_l} 
    \leq
    \min \cbr{ 
    \frac{\sigma_{k-i+1}\rbr{\Ab \rbr{\Ib_n-\wh\Vb_l \wh\Vb_l^*}}}{\sigma_k},
    \frac{\sigma_1\rbr{\Ab \rbr{\Ib_n-\wh\Vb_l \wh\Vb_l^*}}}{\sigma_i} 
    }.
\end{align}
\end{theorem}   
It is worth highlighting that \Cref{eq:residual_based_posterior_left,eq:residual_based_posterior_right} are independent of the unknown true singular subspaces $\Ub_k$ and $\Vb_k$, but only depend on the singular values $\Sigmab_k$ (which may be replaced with the approximations $\wh\Sigmab_l$ from \Cref{algo:rsvd_power_iterations}, \cf~\Cref{subsec:experiment_canonical_angle}) and approximated singular subspaces $\wh\Ub_l$ and $\wh\Vb_l$. Therefore, these upper bounds can be evaluated/approximated efficiently based on the residual of the resulting low-rank approximation $\wh\Ab_l = \wh\Ub_l \wh\Sigmab_l \wh\Vb_l^*$.

\begin{remark}[Left versus right singular subspaces]\label{remark:compare_left_right}
When $\wh\Ub_l$ and $\wh\Vb_l$ consist of approximated left and right singular vectors from \Cref{algo:rsvd_power_iterations}, upper bounds on $\sin\angle_i\rbr{\Vb_k, \wh\Vb_l}$ tend to be smaller than those on $\sin\angle_i\rbr{\Ub_k, \wh\Ub_l}$. This is induced by the algorithmic fact that, in \Cref{algo:rsvd_power_iterations}, $\wh\Vb_l$ is an orthonormal basis of $\Ab^*\Qb_\Xb$, while $\wh\Ub_l$ and $\Qb_\Xb$ are orthonormal bases of $\Xb = \Ab\Omegab$. That is, the evaluation of $\wh\Vb_l$ is enhanced by an additional half power iteration compared with that of $\wh\Ub_l$, which is also reflected by the differences in exponents on $\Sigmab$ (\ie, $2q+1$ versus $2q+2$) in \Cref{thm:space_agnostic_bounds} and \Cref{prop:space_agnostic_estimation}. 
{This difference can be important especially when $q$ is small (\eg, $q=0$). When higher accuracy in the left singular subspace is desirable, one can work with $\Ab^*$.}
\end{remark} 

\begin{proof}[Proof of \Cref{thm:with_oversmp_computable_det}]
Starting with the leading left singular subspace, by definition, for each $i=1,\dots,k$, we have
\begin{align*}
    \sin\angle_i\rbr{\Ub_k, \wh\Ub_l} 
    = &\sigma_{k-i+1}\rbr{\rbr{\Ib_m-\wh\Ub_l \wh\Ub_l^*}\Ub_k}
    \\
    =& \sigma_{k-i+1}\rbr{\rbr{\Ib_m-\wh\Ub_l \wh\Ub_l^*} \Ab \rbr{ \Vb \Sigmab^{-1} \Ub^* \Ub_k } }
    \\
    =& \sigma_{k-i+1}\rbr{ \rbr{\rbr{\Ib_m-\wh\Ub_l \wh\Ub_l^*} \Ab \Vb_k} \Sigmab_k^{-1}}.
\end{align*}
Then, we observe that the following holds simultaneously, 
\begin{align*}
    &\sigma_{k-i+1}\rbr{ \rbr{\rbr{\Ib_m - \wh\Ub_l \wh\Ub_l^*} \Ab \Vb_k} \Sigmab_k^{-1}} \leq \sigma_1\rbr{\rbr{\Ib_m - \wh\Ub_l \wh\Ub_l^*} \Ab \Vb_k} \cdot \sigma_{k-i+1}\rbr{\Sigmab_k^{-1}},
    \\
    &\sigma_{k-i+1}\rbr{ \rbr{\rbr{\Ib_m - \wh\Ub_l \wh\Ub_l^*} \Ab \Vb_k} \Sigmab_k^{-1}} \leq \sigma_{k-i+1}\rbr{\rbr{\Ib_m - \wh\Ub_l \wh\Ub_l^*} \Ab \Vb_k} \cdot \sigma_1\rbr{\Sigmab_k^{-1}},
\end{align*}
where $\sigma_{k-i+1}\rbr{\Sigmab_k^{-1}} = 1/\sigma_i$ and $\sigma_{1}\rbr{\Sigmab_k^{-1}} = 1/\sigma_k$. Finally by \Cref{lemma:Cauchy_interlacing_theorem}, we have
\begin{align*}
    \sigma_{k-i+1}\rbr{\rbr{\Ib_m - \wh\Ub_l \wh\Ub_l^*} \Ab \Vb_k}
    \le 
    \sigma_{k-i+1}\rbr{\rbr{\Ib_m - \wh\Ub_l \wh\Ub_l^*} \Ab}.
\end{align*}

Meanwhile, the upper bound for the leading right singular subspace can be derived analogously by observing that
\begin{align*}
    \sin\angle_i\rbr{\Vb_k, \wh\Vb_l} 
    =&\sigma_{k-i+1}\rbr{ \Vb_k^* \rbr{\Ib_n - \wh\Vb_l\wh\Vb_l^*} } 
    = \sigma_{k-i+1}\rbr{\Sigmab_k^{-1} \Ub_k^* \Ab \rbr{\Ib_n - \wh\Vb_l\wh\Vb_l^*} }.
\end{align*}
\end{proof}

As a potential drawback, although the residuals $\rbr{\Ib_m - \wh\Ub_l \wh\Ub_l^*}\Ab$ and $\Ab\rbr{\Ib_n - \wh\Vb_l \wh\Vb_l^*}$ in \Cref{thm:with_oversmp_computable_det} can be evaluated efficiently in $O(mn)$ and $O(mnl)$ time\footnote{
    For the $O(mn)$ complexity of computing $\rbr{\Ib_m - \wh\Ub_l \wh\Ub_l^*}\Ab$, we assume that $\wh\Ab_l = \wh\Ub_l \wh\Ub_l^* \Ab = \wh\Ub_l \wh\Sigmab_l \wh\Vb_l^*$ is readily available from \Cref{algo:rsvd_power_iterations}. Otherwise (\eg, when \Cref{algo:rsvd_power_iterations} returns $\rbr{\wh\Ub_l, \wh\Sigmab_l, \wh\Vb_l}$ but not $\wh\Ab_l$), the evaluation of $\rbr{\Ib_m - \wh\Ub_l \wh\Ub_l^*}\Ab$ will inevitably take $O(mnl)$ as that of $\Ab\rbr{\Ib_n - \wh\Vb_l \wh\Vb_l^*}$.
}, respectively, the exact evaluation of their full spectra can be unaffordable. A straightforward remedy for this problem is {to use} only the second terms in the right-hand-sides of \Cref{eq:residual_based_posterior_left} and \Cref{eq:residual_based_posterior_right} while estimating $\nbr{\rbr{\Ib_m - \wh\Ub_l \wh\Ub_l^*}\Ab}_2$ and $\nbr{\Ab\rbr{\Ib_n - \wh\Vb_l \wh\Vb_l^*}}_2$ with the randomized power method (\cf \cite{kuczynski1992estimating}, \cite{martinsson2020randomized} Algorithm 4). Alternatively, we leverage the analysis from \cite{nakatsukasa2020sharp} Theorem 6.1 and present the following posterior bounds based only on norms of the residuals which can be estimated efficiently via sampling.

\begin{theorem}\label{thm:separate_gap_bounds}
For any {approximated rank-$l$ SVD} $\Ab \approx \wh\Ub_l \wh\Sigmab_l \wh\Vb_l^*$ (not necessarily obtained by \Cref{algo:rsvd_power_iterations}), recall the notation that $\wh\Ub_l = \sbr{\wh\Ub_k, \wh\Ub_{l \setminus k}}$, $\wh\Vb_l = \sbr{\wh\Vb_k, \wh\Vb_{l \setminus k}}$, while $\wh\Ub_{m \setminus l}, \wh\Ub_{n \setminus l}$ are the orthogonal complements of $\wh\Ub_l, \wh\Vb_l$, respectively.
Then, with $\Eb_{31} \dfeq \wh\Ub^*_{m \setminus l} \Ab \wh\Vb_k$, 
$\Eb_{32} \dfeq \wh\Ub^*_{m \setminus l} \Ab \wh\Vb_{l \setminus k}$, and 
$\Eb_{33} \dfeq \wh\Ub^*_{m \setminus l} \Ab \wh\Vb_{n \setminus l}$, assuming $\sigma_k > \wh\sigma_{k+1}$ and $\sigma_k > \nbr{\Eb_{33}}_2$, we define the spectral gaps
\begin{align*}
    \gamma_1 \dfeq \frac{\sigma_k^2 - \wh\sigma_{k+1}^2}{\sigma_k},
    \quad
    \gamma_2 \dfeq \frac{\sigma_k^2 - \wh\sigma_{k+1}^2}{\wh\sigma_{k+1}},
    \quad
    \Gamma_1 = \frac{\sigma_k^2-\nbr{\Eb_{33}}_2^2}{\sigma_k},
    \quad
    \Gamma_2 = \frac{\sigma_k^2-\nbr{\Eb_{33}}_2^2}{\nbr{\Eb_{33}}_2}
.\end{align*}   
Then for an arbitrary {unitarily} invariant norm $\nnbr{\cdot}$,
\begin{align}
    \label{eq:separate_gap_Uk_Ul}
    &\nnbr{\sin\angle\rbr{\Ub_k,\wh\Ub_l}} \leq \frac{\nnbr{\sbr{\Eb_{31}, \Eb_{32}}}}{\Gamma_1},
    \\
    \label{eq:separate_gap_Vk_Vl}
    &\nnbr{\sin\angle\rbr{\Vb_k,\wh\Vb_l}} \leq \frac{\nnbr{\sbr{\Eb_{31}, \Eb_{32}}}}{\Gamma_2},
\end{align}
and specifically for the spectral or Frobenius norm $\nbr{\cdot}_\xi$ ($\xi=2,F$),
\begin{align}
    \label{eq:separate_gap_Uk_Uk}
    &\nbr{\sin\angle\rbr{\Ub_k,\wh\Ub_k}}_\xi \leq \frac{\nbr{\sbr{\Eb_{31}, \Eb_{32}}}_\xi}{\Gamma_1} 
    \sqrt{1 + \frac{\nbr{\Eb_{32}}_2^2}{\gamma_2^2}},
    \\
    \label{eq:separate_gap_Vk_Vk}
    &\nbr{\sin\angle\rbr{\Vb_k,\wh\Vb_k}}_\xi \leq \frac{\nbr{\sbr{\Eb_{31}, \Eb_{32}}}_\xi}{\Gamma_1} \sqrt{\frac{\nbr{\Eb_{32}}_2^2}{\gamma_1^2} + \frac{\nbr{\Eb_{33}}_2^2}{\sigma_k^2}}.
\end{align}
Furthermore, for all $i \in [k]$,
\begin{align}
    \label{eq:separate_gap_Uk_Ul_anglewise}
    &\sin\angle_i\rbr{\Ub_k,\wh\Ub_l} \leq \frac{\sigma_k}{\sigma_{k-i+1}} \cdot \frac{\nbr{\sbr{\Eb_{31}, \Eb_{32}}}_2}{\Gamma_1},
    \\
    \label{eq:separate_gap_Vk_Vl_anglewise}
    &\sin\angle_i\rbr{\Vb_k,\wh\Vb_l} \leq \frac{\sigma_k}{\sigma_{k-i+1}} \cdot \frac{\nbr{\sbr{\Eb_{31}, \Eb_{32}}}_2}{\Gamma_2},
    \\
    \label{eq:separate_gap_Uk_Uk_anglewise}
    &\sin\angle_i\rbr{\Ub_k,\wh\Ub_k} \leq \frac{\nbr{\sbr{\Eb_{31}, \Eb_{32}}}_2}{\Gamma_1} \sqrt{1 + \rbr{\frac{\sigma_k}{\sigma_{k-i+1}} \cdot \frac{\nbr{\Eb_{32}}_2}{\gamma_2}}^2 },
    \\
    \label{eq:separate_gap_Vk_Vk_anglewise}
    &\sin\angle_i\rbr{\Vb_k,\wh\Vb_k} \leq \frac{\nbr{\sbr{\Eb_{31}, \Eb_{32}}}_2}{\Gamma_1} \sqrt{\rbr{\frac{\sigma_k}{\sigma_{k-i+1}} \cdot \frac{\nbr{\Eb_{32}}_2}{\gamma_1}}^2 + \rbr{\frac{\nbr{\Eb_{33}}_2}{\sigma_k}}^2 }.
\end{align}
\end{theorem}   

In practice, norms of the residuals can be computed as
\begin{align*}
    &\nnbr{\sbr{\Eb_{31},\Eb_{32}}} = \nnbr{\wh\Ub^*_{m \setminus l} \Ab \wh\Vb_l} = \nnbr{\rbr{\Ib_m - \wh\Ub_l \wh\Ub_l^*} \Ab \wh\Vb_l} = \nnbr{\rbr{\Ab - \wh\Ab_l} \wh\Vb_l},
    \\
    &\nbr{\Eb_{32}}_2 = \nbr{\wh\Ub^*_{m \setminus l} \Ab \wh\Vb_{l \setminus k}}_2 = \nbr{\rbr{\Ab - \wh\Ab_l} \wh\Vb_{l \setminus k}}_2,
    \\
    &\nnbr{\Eb_{33}}_2 = \nbr{\wh\Ub^*_{m \setminus l} \Ab \wh\Vb_{n \setminus l}} = \nbr{\rbr{\Ab - \wh\Ab_l} \rbr{\Ib_n - \wh\Vb_l \wh\Vb_l^*}}_2 = \nbr{\Ab - \Ab \wh\Vb_l \wh\Vb_l^*}_2,
\end{align*}    
where the construction of $\rbr{\Ab - \wh\Ab_l} \wh\Vb_l$, $\rbr{\Ab - \wh\Ab_l} \wh\Vb_{l \setminus k}$, and $\Ab - \Ab \wh\Vb_l \wh\Vb_l^*$ takes $O\rbr{mnl}$ time, while the respective norms can be estimated efficiently via sampling (\cf \cite{martinsson2020randomized} Algorithm 1-4, \cite{meyer2020hutch} Algorithm 1-3, etc.). {With an unknown true spectrum in practice, replacing the true singular values in \Cref{thm:separate_gap_bounds} with their corresponding approximations from \Cref{algo:rsvd_power_iterations} usually yields similar upper bounds (\cf~\Cref{subsec:experiment_canonical_angle}).}

The proof of \Cref{thm:separate_gap_bounds} is 
{similar to that of \cite[Thm.~6.1]{nakatsukasa2020sharp}}.  

\begin{proof}[Proof of \Cref{thm:separate_gap_bounds}]
Let $\wt\Ub_{11} \dfeq \wh\Ub^*_{k} \Ub_k$, $\wt\Ub_{21} \dfeq \wh\Ub^*_{l \setminus k} \Ub_k$, $\wt\Ub_{31} \dfeq \wh\Ub^*_{m \setminus l} \Ub_k$, and $\wt\Vb_{11} \dfeq \wh\Vb^*_{k} \Vb_k$, $\wt\Vb_{21} \dfeq \wh\Vb^*_{l \setminus k} \Vb_k$, and $\wt\Vb_{31} \dfeq \wh\Vb^*_{n \setminus l} \Vb_k$.
We start by expressing the canonical angles in terms of $\wt\Ub_{31}$ and $\wt\Ub_{21}$:
\begin{align*}
    \sin\angle\rbr{\Ub_k, \wh\Ub_l} = 
    \sigma\rbr{\wh\Ub^*_{m \setminus l} \Ub_k} = 
    \sigma\rbr{\wt\Ub_{31}},
\end{align*}
\begin{align*}
    \sin\angle\rbr{\Vb_k, \wh\Vb_l} = 
    \sigma\rbr{\wh\Vb^*_{n \setminus l} \Vb_k} = 
    \sigma\rbr{\wt\Vb_{31}},
\end{align*}
\begin{align*}
    \sin\angle\rbr{\Ub_k, \wh\Ub_k} = 
    \sigma\rbr{\bmat{\wh\Ub^*_{l \setminus k} \\ \wh\Ub^*_{m \setminus l}} \Ub_k} = 
    \sigma\rbr{\bmat{\wt\Ub_{21} \\ \wt\Ub_{31}}},
\end{align*}
\begin{align*}
    \sin\angle\rbr{\Vb_k, \wh\Vb_k} = 
    \sigma\rbr{\bmat{\wh\Vb^*_{l \setminus k} \\ \wh\Vb^*_{n \setminus l}} \Vb_k} = 
    \sigma\rbr{\bmat{\wt\Vb_{21} \\ \wt\Vb_{31}}}.
\end{align*}
By observing that for any rank-$l$ approximation in the SVD form $\Ab \approx \wh\Ub_l \wh\Sigmab_l \wh\Vb_l^*$,
\begin{align*}
    \Ab = 
    \wh\Ub_l \wh\Sigmab_l \wh\Vb_l^* + \wh\Ub_{m \setminus l} \wh\Ub_{m \setminus l}^* \Ab = 
    \bmat{\wh\Ub_k & \wh\Ub_{l \setminus k} & \wh\Ub_{m \setminus l}}
    \bmat{\wh\Sigmab_k & \b0 & \b0 \\
    \b0 & \wh\Sigmab_{l \setminus k} & \b0 \\
    \Eb_{31} & \Eb_{32} & \Eb_{33}}
    \bmat{\wh\Vb_k^* \\ \wh\Vb_{l \setminus k}^* \\ \wh\Vb_{n \setminus l}^*}
,\end{align*}  
we left multiply by {$\wh\Ub^* = \sbr{\wh\Ub_k, \wh\Ub_{l \setminus k}, \wh\Ub_{m \setminus l}}^* \in \C^{m \times m}$} and right multiply by $\Vb_k$ on both sides and get
\begin{align}\label{eq:proof_mixed_gap_bound_submatrix_relation_31}
    \bmat{\wt\Ub_{11} \\ \wt\Ub_{21} \\ \wt\Ub_{31}} \Sigmab_k = 
    \bmat{\wh\Sigmab_k & \b0 & \b0 \\
    \b0 & \wh\Sigmab_{l \setminus k} & \b0 \\
    \Eb_{31} & \Eb_{32} & \Eb_{33}}
    \bmat{\wt\Vb_{11} \\ \wt\Vb_{21} \\ \wt\Vb_{31}}
,\end{align}    
while left multiplying $\Ub_k^*$ and right multiplying $\wh\Vb = \sbr{\wh\Vb_k, \wh\Vb_{l \setminus k}, \wh\Vb_{n \setminus l}}$ yield
\begin{align}\label{eq:proof_mixed_gap_bound_submatrix_relation_13}
    \Sigmab_k \bmat{\wt\Vb_{11}^* & \wt\Vb_{21}^* & \wt\Vb_{31}^*} =
    \bmat{\wt\Ub_{11}^* & \wt\Ub_{21}^* & \wt\Ub_{31}^*}
    \bmat{\wh\Sigmab_k & \b0 & \b0 \\
    \b0 & \wh\Sigmab_{l \setminus k} & \b0 \\
    \Eb_{31} & \Eb_{32} & \Eb_{33}}
.\end{align}

\paragraph{Bounding $\sigma\rbr{\wt\Ub_{31}}$ and $\sigma\rbr{\wt\Vb_{31}}$.}
To bound $\sigma\rbr{\wt\Ub_{31}}$, we observe the following from the third row of \Cref{eq:proof_mixed_gap_bound_submatrix_relation_31} and the third column of \Cref{eq:proof_mixed_gap_bound_submatrix_relation_13}, 
\begin{align*}
    \wt\Ub_{31} \Sigmab_k = \Eb_{31} \wt\Vb_{11} + \Eb_{32}\wt\Vb_{21} + \Eb_{33}\wt\Vb_{31},
    \quad
    \wt\Ub_{31}^* \Eb_{33} = \Sigmab_k \wt\Vb_{31}^*.
\end{align*}
Noticing that {$\bmat{\wt\Vb_{11} \\ \wt\Vb_{21}} = \wh\Vb_l^* \Vb_k$} and $\nbr{\wh\Vb_l^* \Vb_k}_2 \le 1$, we have
\begin{align*}
    \nnbr{\wt\Ub_{31} \Sigmab_k} 
    \le &\nnbr{\sbr{\Eb_{31},\Eb_{32}}} \nbr{\wh\Vb_l^* \Vb_k}_2 + \nnbr{\Eb_{33} \Eb_{33}^* \wt\Ub_{31} \Sigmab_k^{-1}}
    \\
    \le &\nnbr{\sbr{\Eb_{31},\Eb_{32}}} + \frac{\nbr{\Eb_{33}}_2^2}{\sigma_k^2} \nnbr{\wt\Ub_{31} \Sigmab_k}
\end{align*} 
for all $i \in [\min\rbr{k,m-l}]$, which implies that
\begin{align*}
    \nnbr{\wt\Ub_{31} \Sigmab_k} 
    \le \rbr{1-\frac{\nbr{\Eb_{33}}_2^2}{\sigma_k^2}}^{-1} \nnbr{\sbr{\Eb_{31},\Eb_{32}}}
    = \sigma_k \cdot \frac{\nnbr{\sbr{\Eb_{31},\Eb_{32}}}}{\Gamma_1},
\end{align*}
and leads to 
\begin{align*}
    &\nnbr{\wt\Ub_{31}} 
    \le \frac{1}{\sigma_k} \nnbr{\wt\Ub_{31} \Sigmab_k} 
    \le \frac{\nnbr{\sbr{\Eb_{31},\Eb_{32}}}}{\Gamma_1},
    \\
    &\sigma_i\rbr{\wt\Ub_{31}}
    \le \frac{1}{\sigma_{k-i+1}} \nbr{\wt\Ub_{31} \Sigmab_k}_2 
    \le \frac{\sigma_k}{\sigma_{k-i+1}} \cdot \frac{\nbr{\sbr{\Eb_{31}, \Eb_{32}}}_2}{\Gamma_1}
    \quad \forall~i \in [k],
\end{align*}
where the second line follows from \Cref{lemma:individual_sval_of_product}. These lead to \Cref{eq:separate_gap_Uk_Ul,eq:separate_gap_Uk_Ul_anglewise}

To bound $\sigma\rbr{\wt\Vb_{31}}$, we use the relation $\wt\Ub_{31}^* \Eb_{33} = \Sigmab_k \wt\Vb_{31}^*$,
\begin{align*}
    &\nnbr{\wt\Vb_{31}} \le 
    \frac{\nbr{\Eb_{33}}_2}{\sigma_k} \nnbr{\wt\Ub_{31}} \le 
    \frac{\nnbr{\sbr{\Eb_{31}, \Eb_{32}}}}{\Gamma_2},
    \\
    &\sigma_i\rbr{\wt\Vb_{31}} 
    \le \frac{1}{\sigma_k} \sigma_i\rbr{\Eb_{33}^* \wt\Ub_{31}}
    \le \frac{\nbr{\Eb_{33}}_2}{\sigma_k} \sigma_i\rbr{\wt\Ub_{31}}
    \le \frac{\sigma_k}{\sigma_{k-i+1}} \cdot \frac{\nbr{\sbr{\Eb_{31}, \Eb_{32}}}_2}{\Gamma_2}
    \quad \forall~i \in[k].
\end{align*}    
We therefore have upper bounds \Cref{eq:separate_gap_Vk_Vl,eq:separate_gap_Vk_Vl_anglewise}.

\paragraph{Bounding $\sigma\rbr{\wt\Ub_{21}}$ and $\sigma\rbr{\wt\Vb_{21}}$.}
To bound $\sigma\rbr{\wt\Ub_{21}}$, we leverage the second row of \Cref{eq:proof_mixed_gap_bound_submatrix_relation_31} and the second column of \Cref{eq:proof_mixed_gap_bound_submatrix_relation_13},
\begin{align*}
    \wt\Ub_{21} \Sigmab_k = \wh\Sigmab_{l \setminus k} \wt\Vb_{21}, 
    \quad
    \Sigmab_k \wt\Vb_{21}^* = \wt\Ub_{21}^* \wh\Sigmab_{l \setminus k} + \wt\Ub_{31}^* \Eb_{32}.
\end{align*}    
Up to rearrangement, we observe that
\begin{align*}
    \nnbr{\wt\Ub_{21} \Sigmab_k} =
    & \nnbr{\wh\Sigmab_{l \setminus k} \rbr{\wh\Sigmab_{l \setminus k} \wt\Ub_{21} + \Eb_{32}^* \wt\Ub_{31}} \Sigmab_k^{-1}}
    \\
    \le &\frac{\nbr{\wh\Sigmab_{l \setminus k}}_2^2}{\sigma_k^2} \nnbr{\wt\Ub_{21} \Sigmab_k} + \frac{\nbr{\wh\Sigmab_{l \setminus k}}_2}{\sigma_k} \nnbr{\Eb_{32}^* \wt\Ub_{31}},
\end{align*}
which implies that
\begin{align*}
    &\nnbr{\wt\Ub_{21} \Sigmab_k} 
    \le \rbr{1-\frac{\wh\sigma_{k+1}^2}{\sigma_k^2}}^{-1} \frac{\wh\sigma_{k+1}}{\sigma_k} \nnbr{\Eb_{32}^* \wt\Ub_{31}}
    \le \sigma_k \cdot \frac{\nbr{\Eb_{32}}_2}{\gamma_2} \nnbr{\wt\Ub_{31}},
\end{align*}
and therefore, with \Cref{lemma:individual_sval_of_product}, for all $i \in [k]$,
\begin{align*}
    &\nnbr{\wt\Ub_{21}}
    \le \frac{1}{\sigma_k} \nnbr{\wt\Ub_{21} \Sigmab_k} 
    \le \frac{\nbr{\Eb_{32}}_2}{\gamma_2} \nnbr{\wt\Ub_{31}},
    \\
    &\sigma_i\rbr{\wt\Ub_{21}} 
    \le \frac{1}{\sigma_{k-i+1}} \nbr{\wt\Ub_{21} \Sigmab_k}_2
    \le \frac{\sigma_k}{\sigma_{k-i+1}} \cdot \frac{\nbr{\Eb_{32}}_2}{\gamma_2} \nbr{\wt\Ub_{31}}_2.
\end{align*}
Then, with the stronger inequality for the spectral or Frobenius norm $\nbr{\cdot}_{\xi}$ ($\xi=2,F$), 
\begin{align*}
    \nbr{\bmat{\wt\Ub_{21} \\ \wt\Ub_{31}}}_\xi 
    \le &\sqrt{\nbr{\wt\Ub_{31}}_\xi^2 + \nbr{\wt\Ub_{21}}_\xi^2} 
    \le \nbr{\wt\Ub_{31}}_\xi \sqrt{1 + \frac{\nbr{\Eb_{32}}_2^2}{\gamma_2^2}}
    \\
    \le &\frac{\nbr{\sbr{\Eb_{31}, \Eb_{32}}}_\xi}{\Gamma_1} 
    \sqrt{1 + \frac{\nbr{\Eb_{32}}_2^2}{\gamma_2^2}}
\end{align*} 
leads to \Cref{eq:separate_gap_Uk_Uk}.
Meanwhile for \Cref{eq:separate_gap_Uk_Uk_anglewise}, the individual canonical angles are upper bounded by
\begin{align*}
    \sigma_i\rbr{\bmat{\wt\Ub_{21} \\ \wt\Ub_{31}}} 
    = &\sqrt{\sigma_i\rbr{\wt\Ub_{21}^* \wt\Ub_{21} + \wt\Ub_{31}^* \wt\Ub_{31}}}
    \\
    \rbr{\text{\Cref{lemma:individual_sval_of_sum}}} \quad
    \le &\sqrt{\nbr{\wt\Ub_{31}}_2^2 + \sigma_i\rbr{\wt\Ub_{21}}^2}
    \\
    \le &\nbr{\wt\Ub_{31}}_2 \sqrt{1 + \rbr{\frac{\sigma_k}{\sigma_{k-i+1}} \cdot \frac{\nbr{\Eb_{32}}_2}{\gamma_2}}^2 }
    \\
    \le &\frac{\nbr{\sbr{\Eb_{31}, \Eb_{32}}}_2}{\Gamma_1} \sqrt{1 + \rbr{\frac{\sigma_k}{\sigma_{k-i+1}} \cdot \frac{\nbr{\Eb_{32}}_2}{\gamma_2}}^2 }.
\end{align*}
Analogously, by observing that 
\begin{align*}
    \nnbr{\wt\Vb_{21} \Sigmab_k} =
    & \nnbr{\wh\Sigmab_{l \setminus k}^2 \wt\Vb_{21} \Sigmab_k^{-1} + \Eb_{32}^* \wt\Ub_{31}}
    \le \frac{\nbr{\wh\Sigmab_{l \setminus k}}_2^2}{\sigma_k^2} \nnbr{\wt\Vb_{21} \Sigmab_k} + \nnbr{\Eb_{32}^* \wt\Ub_{31}},
\end{align*}
we have that, by \Cref{lemma:individual_sval_of_product}, for all $i \in [k]$,
\begin{align*}
    &\nnbr{\wt\Vb_{21} \Sigmab_k} 
    \le \rbr{1-\frac{\wh\sigma_{k+1}^2}{\sigma_k^2}}^{-1} \nnbr{\Eb_{32}^* \wt\Ub_{31}}
    \le \sigma_k \cdot \frac{\nbr{\Eb_{32}}_2}{\gamma_1} \nnbr{\wt\Ub_{31}},
    \\
    &\nnbr{\wt\Vb_{21}} 
    \le \frac{1}{\sigma_k} \nnbr{\wt\Vb_{21} \Sigmab_k} 
    \le \frac{\nbr{\Eb_{32}}_2}{\gamma_1} \nnbr{\wt\Ub_{31}},
    \\
    &\sigma_i\rbr{\wt\Vb_{21}}
    \le \frac{1}{\sigma_{k-i+1}} \nbr{\wt\Vb_{21} \Sigmab_k}_2
    \le \frac{\sigma_k}{\sigma_{k-i+1}} \cdot \frac{\nbr{\Eb_{32}}_2}{\gamma_1} \nbr{\wt\Ub_{31}}_2,
\end{align*}
and therefore for the spectral or Frobenius norm $\nbr{\cdot}_{\xi}$ ($\xi=2,F$),
\begin{align*}
    \nbr{\bmat{\wt\Vb_{21} \\ \wt\Vb_{31}}}_\xi
    \le &\sqrt{\nbr{\wt\Vb_{21}}_\xi^2 + \nbr{\wt\Vb_{31}}_\xi^2}  
    \le \nbr{\wt\Ub_{31}}_\xi \sqrt{\frac{\nbr{\Eb_{32}}_2^2}{\gamma_1^2} + \frac{\nbr{\Eb_{33}}_2^2}{\sigma_k^2}}
    \\
    \le &\frac{\nbr{\sbr{\Eb_{31}, \Eb_{32}}}_\xi}{\Gamma_1} \sqrt{\frac{\nbr{\Eb_{32}}_2^2}{\gamma_1^2} + \frac{\nbr{\Eb_{33}}_2^2}{\sigma_k^2}},
\end{align*}  
which leads to \Cref{eq:separate_gap_Vk_Vk}.
Additionally for individual canonical angles $i \in [k]$,
\begin{align*}
    \sigma_i\rbr{\bmat{\wt\Vb_{21} \\ \wt\Vb_{31}}} 
    \le &\sqrt{\sigma_i\rbr{\wt\Vb_{21}}^2 + \nbr{\wt\Vb_{31}}_2^2}
    \quad \rbr{\text{\Cref{lemma:individual_sval_of_sum}}}
    \\
    \le &\nbr{\wt\Ub_{31}}_2 \sqrt{\rbr{\frac{\sigma_k}{\sigma_{k-i+1}} \cdot \frac{\nbr{\Eb_{32}}_2}{\gamma_1}}^2 + \rbr{\frac{\nbr{\Eb_{33}}_2}{\sigma_k}}^2 }
    \\
    \le &\frac{\nbr{\sbr{\Eb_{31}, \Eb_{32}}}_2}{\Gamma_1} \sqrt{\rbr{\frac{\sigma_k}{\sigma_{k-i+1}} \cdot \frac{\nbr{\Eb_{32}}_2}{\gamma_1}}^2 + \rbr{\frac{\nbr{\Eb_{33}}_2}{\sigma_k}}^2 }.
\end{align*}
{This yields \Cref{eq:separate_gap_Vk_Vk_anglewise} and completes the proof.}
\end{proof}

\section{Numerical Experiments}\label{sec:experiments}

First, we present numerical comparisons among different canonical angle upper bounds and the unbiased estimates on the left and right leading singular subspaces of various synthetic and real data matrices. We start by describing the target matrices in \Cref{subsec:target_matrix}.
In \Cref{subsec:experiment_canonical_angle}, we discuss the performance of the unbiased estimates, as well as the relative tightness of the canonical angle bounds, for different algorithmic choices based on the numerical observations.
Second, in \Cref{subsec:balance_oversampling_power_iterations_example}, we present an illustrative example that {provides} insight into the balance between oversampling and power iterations brought by the space-agnostic bounds.

\subsection{Target Matrices}\label{subsec:target_matrix}
We consider several different classes of target matrices, including some synthetic random matrices with 
different spectral patterns, as well as an empirical dataset, as summarized below:
\begin{enumerate}
    \item A random sparse non-negative (SNN) matrix~\cite{sorensen2016deim} $\Ab$ of size $m \times n$ takes the form,
    \begin{equation}
        \label{eq:snn-def}
        \Ab = \text{SNN}\rbr{a,r_1}:= \sum_{i=1}^{r_1} \frac{a}{i} \xb_i \yb_i^T + \sum_{i=r_1+1}^{\min\rbr{m,n}} \frac{1}{i} \xb_i \yb_i^T
    \end{equation}
    where $a>1$ and $r_1 < \min\rbr{m,n}$ control the spectral decay, and $\xb_i \in \C^m$, $\yb_i \in \C^{n}$ are random sparse vectors with non-negative entries.
    In the experiments, we test on two random SNN matrices of size $500 \times 500$ with $r_1 = 20$ and $a=1,100$, respectively.
    
    \item Gaussian dense matrices with controlled spectral decay are randomly generated via {a} similar construction {to} the SNN matrix, with $\xb_j \in \SSS^{m-1}$ and $\yb_j \in \SSS^{n-1}$ in (\ref{eq:snn-def}) replaced by uniformly random dense orthonormal vectors. The generating procedures for $\Ab \in \C^{m \times n}$ with rank $r \leq \min\rbr{m,n}$ can be summarized as following:
    \begin{enumerate}[label=(\roman*)]
        \item Draw Gaussian random matrices, $\Gb_m \in \C^{m \times r}$ and $\Gb_n \in \C^{m \times r}$.
        \item Compute $\Ub = \text{ortho}(\Gb_m) \in \C^{m \times r}$, $\Vb = \text{ortho}(\Gb_n) \in \C^{n \times r}$ as orthonormal bases. 
        \item Given the spectrum $\Sigmab = \diag\rbr{\sigma_1,\dots,\sigma_r}$, we construct $\Ab = \Ub \Sigmab \Vb^*$.
    \end{enumerate}
    In the experiments, we consider two types of spectral decay: 
    \begin{enumerate}[label=(\roman*)]
        \item slower decay with $r_1=20$, $\sigma_i = 1$ for all $i \in [r_1]$, $\sigma_i=1/\sqrt{i-r_1+1}$ for all $i=r_1+1,\dots,r$, and
        \item faster decay with $r_1=20$, $\sigma_i = 1$ for all $i \in [r_1]$, $\sigma_i=\max(0.99^{i-r_1}, 10^{-3})$ for all $i=r_1+1,\dots,r$.
    \end{enumerate}
    
    \item MNIST training set consists of $60,000$ images of hand-written digits from $0$ to $9$. Each image is of size $28 \times 28$. We form the target matrices by uniformly sampling $N = 800$ images from the MNIST training set. The images are flattened and normalized to form a full-rank matrix of size $N \times d$ where $d = 784$ is the size of the flattened images, with entries bounded in $[0,1]$. The nonzero entries take approximately $20\%$ of the matrix for both the training and the testing sets. 
\end{enumerate}

\subsection{Canonical Angle Bounds and Estimates}\label{subsec:experiment_canonical_angle}

Now we present numerical comparisons of the performance of the canonical angle bounds and the unbiased estimates under different algorithmic choices.
Considering the scenario where the true matrix spectra may not be available in practice, we calculate two sets of upper bounds, one from the true spectra $\Sigmab \in \C^{r \times r}$ and the other from the $l$ approximated singular values from \Cref{algo:rsvd_power_iterations}. 
For the later, we pad the approximated spectrum $\wh\Sigmab_l=\diag\rbr{\wh\sigma_1,\dots,\wh\sigma_l}$ with {$r-l$ copies of} $\wh\sigma_l$ and evaluate the canonical angle bounds and estimates with $\wt\Sigmab=\diag\rbr{\wh\sigma_1,\dots,\wh\sigma_l,\dots,\wh\sigma_l} \in \C^{r \times r}$.

\begin{figure}[!ht]
    \centering
    \includegraphics[width=\linewidth]{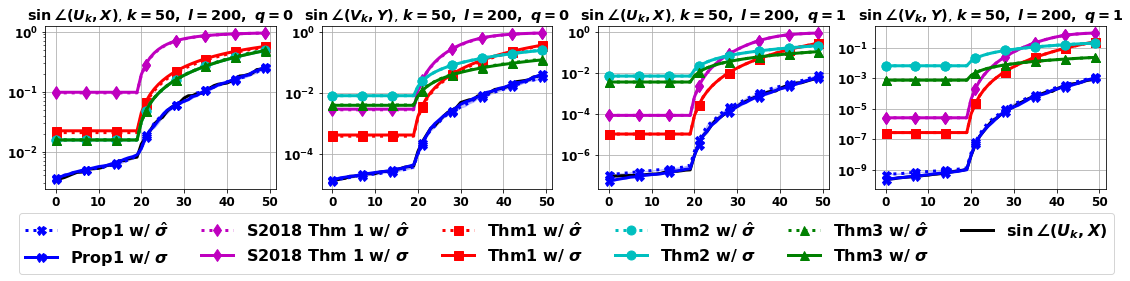}
    \caption{Synthetic Gaussian with the slower spectral decay. $k=50$, $l=200$, $q=0,1$.}
    \label{fig:Gaussian-poly1-kl_k50_l200}
\end{figure}

\begin{figure}[!ht]
    \centering
    \includegraphics[width=\linewidth]{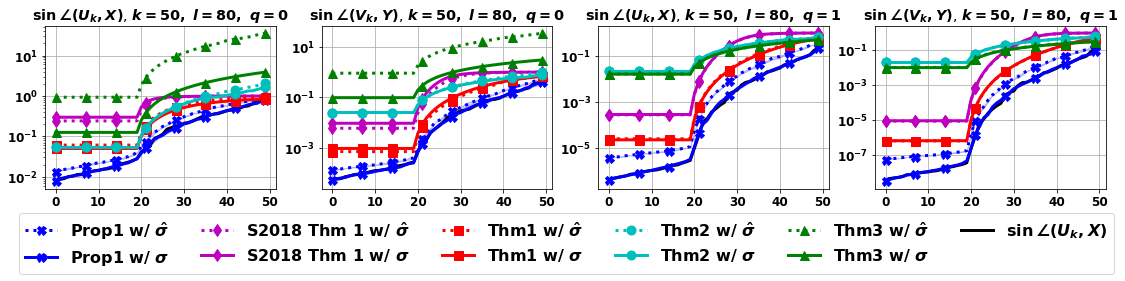}
    \caption{Synthetic Gaussian with the slower spectral decay. $k=50$, $l=80$, $q=0,1$.}
    \label{fig:Gaussian-poly1-kl_k50_l80}
\end{figure}

\begin{figure}[!ht]
    \centering
    \includegraphics[width=\linewidth]{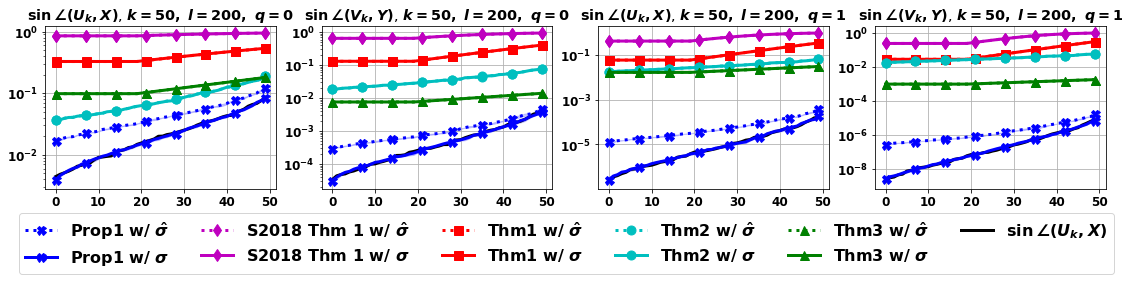}
    \caption{Synthetic Gaussian with the faster spectral decay. $k=50$, $l=200$, $q=0,1$.}
    \label{fig:Gaussian-exp-kl_k50_l200}
\end{figure}

\begin{figure}[!ht]
    \centering
    \includegraphics[width=\linewidth]{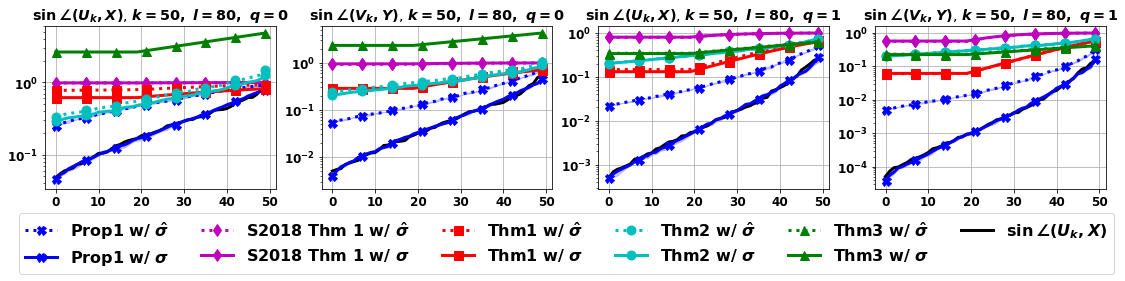}
    \caption{Synthetic Gaussian with the faster spectral decay. $k=50$, $l=80$, $q=0,1$.}
    \label{fig:Gaussian-exp-kl_k50_l80}
\end{figure}

\begin{figure}[!ht]
    \centering
    \includegraphics[width=\linewidth]{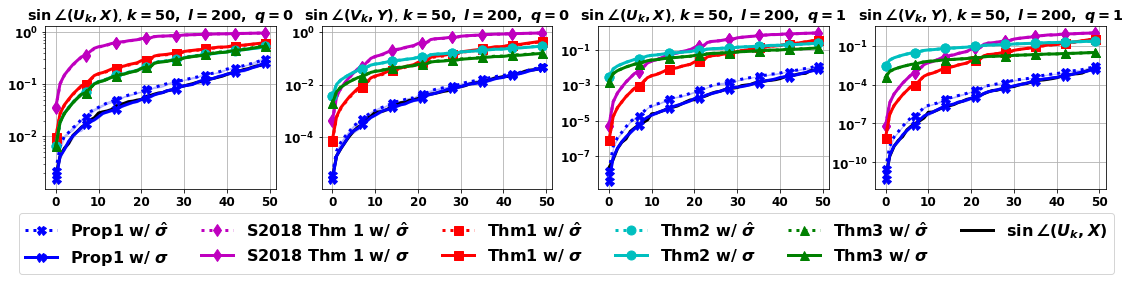}
    \caption{SNN with $r_1=20$, $a=1$. $k=50$, $l=200$, $q=0,1$.}
    \label{fig:SNN-m500n500r20a1-kl_k50_l200}
\end{figure}

\begin{figure}[!ht]
    \centering
    \includegraphics[width=\linewidth]{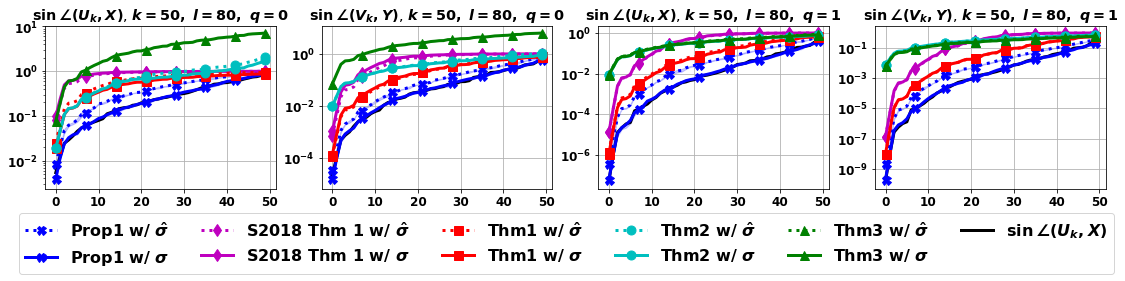}
    \caption{SNN with $r_1=20$, $a=1$. $k=50$, $l=80$, $q=0,1$.}
    \label{fig:SNN-m500n500r20a1-kl_k50_l80}
\end{figure}

\begin{figure}[!ht]
    \centering
    \includegraphics[width=\linewidth]{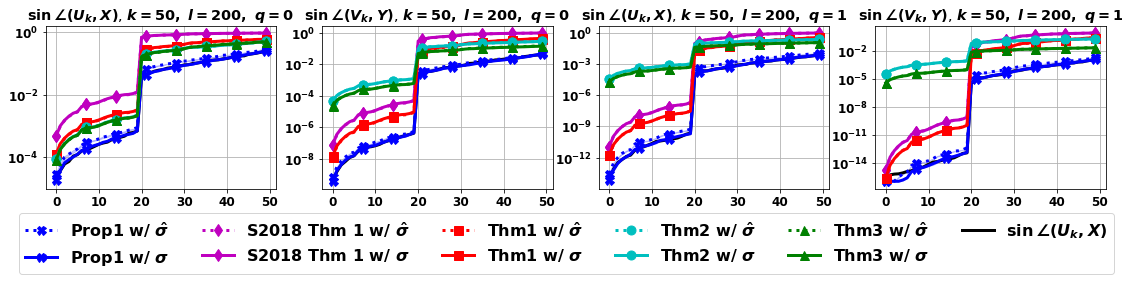}
    \caption{SNN with $r_1=20$, $a=100$. $k=50$, $l=200$, $q=0,1$.}
    \label{fig:SNN-m500n500r20a100-kl_k50_l200}
\end{figure}

\begin{figure}[!ht]
    \centering
    \includegraphics[width=\linewidth]{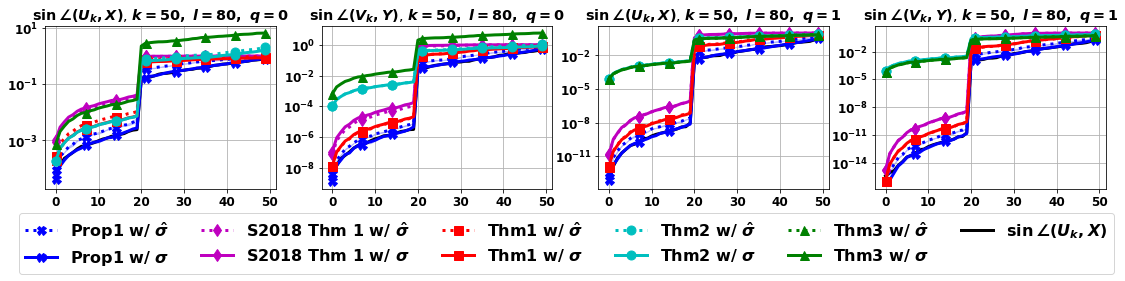}
    \caption{SNN with $r_1=20$, $a=100$. $k=50$, $l=80$, $q=0,1$.}
    \label{fig:SNN-m500n500r20a100-kl_k50_l80}
\end{figure}

\begin{figure}[!ht]
    \centering
    \includegraphics[width=\linewidth]{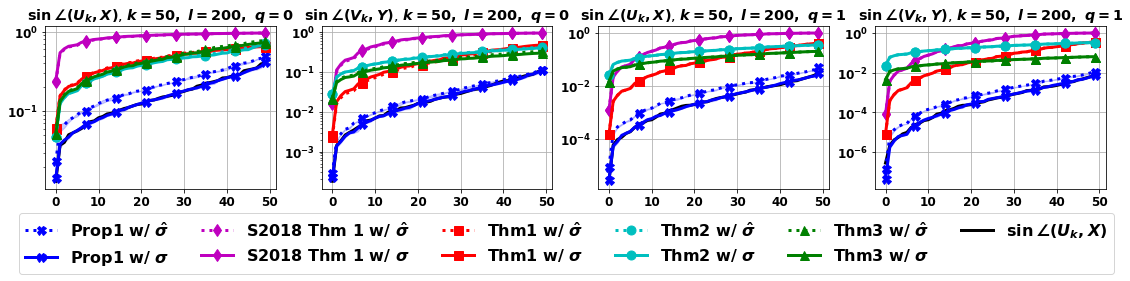}
    \caption{$800$ randomly sampled images from the MNIST training set. $k=50$, $l=200$, $q=0,1$.}
    \label{fig:mnist-train800-kl_k50_l200}
\end{figure}

\begin{figure}[!ht]
    \centering
    \includegraphics[width=\linewidth]{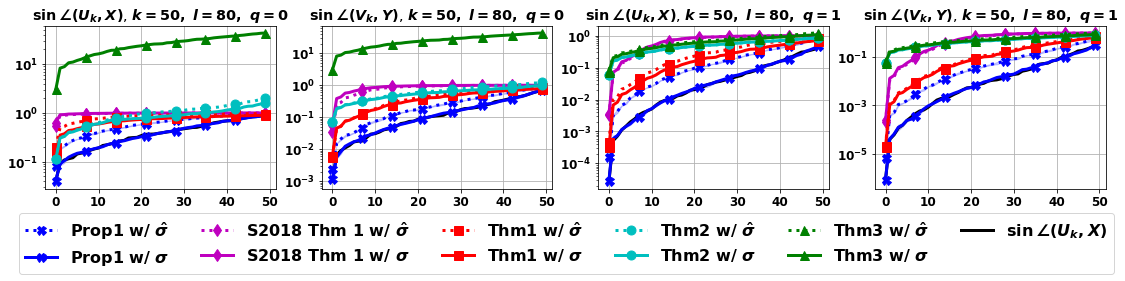}
    \caption{$800$ randomly sampled images from the MNIST training set. $k=50$, $l=80$, $q=0,1$.}
    \label{fig:mnist-train800-kl_k50_l80}
\end{figure}

\begin{figure}[!ht]
    \centering
    \includegraphics[width=\linewidth]{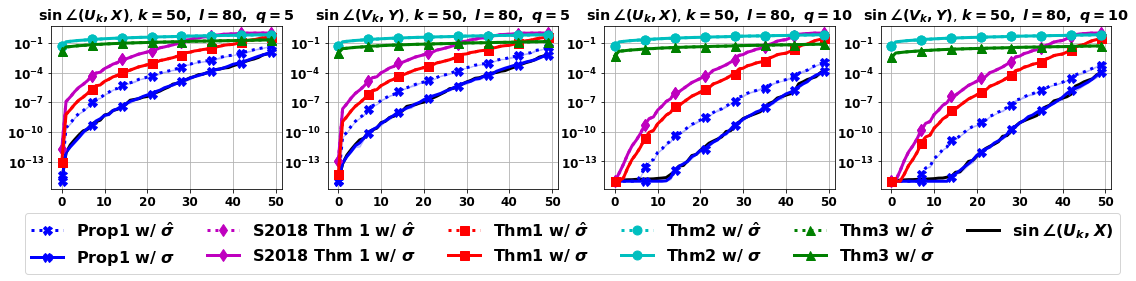}
    \caption{$800$ randomly sampled images from the MNIST training set. $k=50$, $l=80$, $q=5,10$.}
    \label{fig:mnist-train800-kl_k50_l80_p5-10}
\end{figure}

From \Cref{fig:Gaussian-poly1-kl_k50_l200} to \Cref{fig:mnist-train800-kl_k50_l80_p5-10},
\begin{enumerate}
    \item {
        Red lines and dashes (\b{Thm1 w/ $\sigmab$ and $\wh\sigmab$})
    } represent the space-agnostic probabilistic bounds in \Cref{thm:space_agnostic_bounds} evaluated with the true (lines) and approximated (dashes) singular values, $\Sigmab$, and $\wt\Sigmab$, respectively, where we simply ignore tail decay and suppress constants for the distortion factors and set $\eps_1 = \sqrt{\frac{k}{l}}$ and $\eps_2=\sqrt{\frac{l}{r-k}}$ in {\Cref{eq:space_agnostic_left,eq:space_agnostic_right}};
    
    \item {
        Blue lines and dashes (\b{Prop1 w/ $\sigmab$ and $\wh\sigmab$})
    } represent the unbiased space-agnostic estimates in \Cref{prop:space_agnostic_estimation} (averages of $N=3$ independent trials with 
    {blue shades} marking the corresponding minima and maxima in the trials) evaluated with the true (lines) and approximated (dashes) singular values, $\Sigmab$ and $\wt\Sigmab$, respectively;
    
    \item {
        Cyan lines and dashes (\b{Thm2 w/ $\sigmab$ and $\wh\sigmab$})
    } represent the posterior residual-based bounds in \Cref{thm:with_oversmp_computable_det} evaluated with the true (lines) and approximated (dashes) singular values, $\Sigmab$, and $\wt\Sigmab$, respectively;
    
    \item {
        Green lines and dashes (\b{Thm3 w/ $\sigmab$ and $\wh\sigmab$})
    } represent the posterior residual-based bounds \Cref{eq:separate_gap_Uk_Ul} and \Cref{eq:separate_gap_Vk_Vl} in \Cref{thm:separate_gap_bounds} evaluated with the true (lines) and approximated (dashes) singular values, $\Sigmab$, and $\wt\Sigmab$, respectively;
    
    \item {
        Magenta lines and dashes (\b{S2018 Thm1 w/ $\sigmab$ and $\wh\sigmab$})
    } represent the upper bounds in \cite{saibaba2018randomized} Theorem 1 (\ie, \Cref{eq:saibaba2018_thm1_left} and \Cref{eq:saibaba2018_thm1_right}) evaluated with the true (lines) and approximated (dashes) singular values, $\Sigmab$ and $\wt\Sigmab$, respectively, and the unknown true singular subspace such that $\Omegab_1 = \Vb_k^*\Omegab$ and $\Omegab_2 = \Vb_{r \setminus k}^*\Omegab$;
    
    \item {
        Black lines
    } mark the true canonical angles $\sin\angle\rbr{\Ub_k, \wh\Ub_l}$.
\end{enumerate}

We recall from \Cref{remark:compare_left_right} that, by the algorithmic construction of \Cref{algo:rsvd_power_iterations}, for given $q$, canonical angles of the right singular spaces $\sin\angle\rbr{\Vb_k, \wh\Vb_l}$ are evaluated with half more power iterations than those of the left singular spaces $\sin\angle\rbr{\Ub_k, \wh\Ub_l}$. 
That is, $\sin\angle\rbr{\Ub_k, \wh\Ub_l}$, $\sin\angle\rbr{\Vb_k, \wh\Vb_l}$ with $q=0,1$ in \Cref{fig:Gaussian-poly1-kl_k50_l200}-\Cref{fig:mnist-train800-kl_k50_l80} can be viewed as canonical angles of randomized subspace approximation with $q=0,0.5,1,1.5$ power iterations, respectively; while \Cref{fig:mnist-train800-kl_k50_l80_p5-10} corresponds to randomized subspace approximations constructed with $q=5,5.5,10,10.5$ power iterations analogously.

For each set of upper bounds/unbiased estimates, we observe the following.
\begin{enumerate}
    \item The {
        space-agnostic probabilistic bounds (\b{Thm1 w/ $\sigmab$ and $\wh\sigmab$})
    } in \Cref{thm:space_agnostic_bounds} provide tighter statistical guarantees for the canonical angles of all the tested target matrices in comparison to those from {
        \cite{saibaba2018randomized} Theorem 1 (\b{S2018 Thm1 w/ $\sigmab$ and $\wh\sigmab$})
    }, as explained in \Cref{remark:prior_bound_compare_exist}.
    
    \item The 
    {
    unbiased estimators (\b{Prop1 w/ $\sigmab$ and $\wh\sigmab$})} in \Cref{prop:space_agnostic_estimation} yield accurate approximations for the true canonical angles on all the tested target matrices with as few as $N=3$ trials, while enjoying good empirical concentrate. As a potential drawback, the accuracy of the unbiased estimates may be compromised when approaching the machine epsilon (as observed in \Cref{fig:SNN-m500n500r20a100-kl_k50_l200}, $\sin\angle\rbr{\Vb_k,\Yb},~q=1$).

    \item The {
        posterior residual-based bounds (\b{Thm2 w/ $\sigmab$ and $\wh\sigmab$})
    } in \Cref{thm:with_oversmp_computable_det} are relatively tighter among the compared bounds in the setting with larger oversampling ($l=4k$), and no power iterations ($\sin\angle\rbr{\Ub_k,\wh\Ub_l}$ with $q=0$) or exponential spectral decay (\Cref{fig:Gaussian-exp-kl_k50_l200})
    
    \item The {
        posterior residual-based bounds \Cref{eq:separate_gap_Uk_Ul,eq:separate_gap_Vk_Vl} (\b{Thm3 w/ $\sigmab$ and $\wh\sigmab$})
    } in \Cref{thm:separate_gap_bounds} share the similar relative tightness as the posterior residual-based bounds in \Cref{thm:with_oversmp_computable_det}, but are slightly more sensitive to power iterations. As shown in \Cref{fig:Gaussian-exp-kl_k50_l200}, on a target matrix with exponential spectral decay and large oversampling ($l=4k$), \Cref{thm:separate_gap_bounds} gives tighter posterior guarantees when $q>0$. However, with the addition assumptions $\sigma_k > \wh\sigma_{k+1}$ and $\sigma_k > \nbr{\Eb_{33}}_2$, \Cref{thm:separate_gap_bounds} usually requires large oversampling ($l=4k$) in order to provide non-trivial (\ie, {within} the range $[0,1]$) bounds.
\end{enumerate}

For target matrices with various patterns of spectral decay, with different combinations of oversampling ($l=1.6k,4k$) and power iterations ($q=0,1$), we make the following observations on the relative tightness of upper bounds in \Cref{thm:space_agnostic_bounds}, \Cref{thm:with_oversmp_computable_det}, and \Cref{thm:separate_gap_bounds}.
\begin{enumerate}
    \item For target matrices with subexponential spectral decay, the space-agnostic bounds in \Cref{thm:space_agnostic_bounds} are relatively tighter in most tested settings, except for the setting in \Cref{fig:Gaussian-exp-kl_k50_l200} with larger oversampling ($l=200$) and no power iterations ($q=0$).
    
    \item For target matrices with exponential spectral decay (\Cref{fig:Gaussian-exp-kl_k50_l200} and \Cref{fig:Gaussian-exp-kl_k50_l80}), the posterior residual-based bounds in \Cref{thm:with_oversmp_computable_det} and \Cref{thm:separate_gap_bounds} tend to be relatively tighter, especially with large oversampling (\Cref{fig:Gaussian-exp-kl_k50_l200} with $l=4k$). Meanwhile, with power iterations $q>0$, \Cref{thm:separate_gap_bounds} tend to be tighter than \Cref{thm:with_oversmp_computable_det}.
\end{enumerate} 

Furthermore, considering the scenario with an unknown true spectrum $\Sigmab$, we plot estimations for the upper bounds in \Cref{thm:space_agnostic_bounds}, \Cref{thm:with_oversmp_computable_det}, \Cref{thm:separate_gap_bounds}, and the unbiased estimates in \Cref{prop:space_agnostic_estimation}, evaluated with a padded approximation of the spectrum $\wt\Sigmab=\diag\rbr{\wh\sigma_1,\dots,\wh\sigma_l,\dots,\wh\sigma_l}$, which leads to mild overestimations, as marked in dashes from \Cref{fig:Gaussian-poly1-kl_k50_l200} to \Cref{fig:mnist-train800-kl_k50_l80_p5-10}.

\subsection{Balance between Oversampling and Power Iterations}\label{subsec:balance_oversampling_power_iterations_example}

To illustrate the insight cast by \Cref{thm:space_agnostic_bounds} on the balance between oversampling and power iterations, we consider the following synthetic example.
\begin{example}
Given a target rank $k \in \N$, we consider a simple synthetic matrix $\Ab \in \C^{r \times r}$ of size $r = (1+\beta)k$, consisting of random singular subspaces (generated by orthonormalizing Gaussian matrices) and a step spectrum:
\begin{align*}
    \sigma\rbr{\Ab}=\diag(\underbrace{\sigma_1,\dots,\sigma_1}_{\sigma_i=\sigma_1 ~\forall~ i \le k}, \underbrace{\sigma_{k+1},\dots,\sigma_{k+1}}_{\sigma_i=\sigma_{k+1} ~\forall~ i \ge k+1}).
\end{align*}
We fix a budget of $N=\alpha k$ matrix-vector multiplications with $\Ab$ in total. The goal is to distribute the computational budget between the sample size $l$ and the number of power iterations $q$ for the smaller canonical angles $\angle\rbr{\Ub_k, \wh\Ub_l}$.

Leveraging \Cref{thm:space_agnostic_bounds}, we start by fixing $\gamma>1$ associated with the constants $\eps_1=\gamma \sqrt{k/l}$ and $\eps_2=\gamma \sqrt{l/(r-k)}$ in \Cref{eq:space_agnostic_left} such that $l \ge \gamma^2 k$ and $2q+1 < \alpha/\gamma^2$. Characterized by $\gamma$, the right-hand-side of \Cref{eq:space_agnostic_left} under fixed budget $N$ (\ie, $N \ge l(2q+1)$) is defined as:
\begin{align}\label{eq:space_agnostic_example_rhs}
    \phi_\gamma\rbr{q} \dfeq &\rbr{1 + \frac{1-\eps_1}{1+\eps_2} \cdot \frac{l}{r-k} \rbr{\frac{\sigma_1}{\sigma_{k+1}}}^{4q+2}}^{-\frac{1}{2}}
    \\
    = &\rbr{1 + \frac{\alpha-\gamma\sqrt{\alpha (2q+1)}}{\beta(2q+1)+\gamma\sqrt{\alpha\beta(2q+1)}} \rbr{\frac{\sigma_1}{\sigma_{k+1}}}^{4q+2}}^{-\frac{1}{2}}. \nonumber
\end{align}
With the synthetic step spectrum, the dependence of \Cref{eq:space_agnostic_left} on $\sigma\rbr{\Ab}$ is reduced to the spectral gap $\sigma_1/\sigma_{k+1}$ in \Cref{eq:space_agnostic_example_rhs}. 

As a synopsis, \Cref{tab:space_agnostic_example_parameter_summary} summarizes the relevant parameters that characterize the problem setup.
\begin{table}[!h]
    \centering
    \caption{Given $\Ab \in \C^{r \times r}$ with a spectral gap $\sigma_1/\sigma_{k+1}$, a target rank $k$, and a budget of $N$ matrix-vector multiplications, we consider applying \Cref{algo:rsvd_power_iterations} with a sample size $l$ and $q$ power iterations.}
    \label{tab:space_agnostic_example_parameter_summary}
    \begin{tabular}{c|c|c}
    \toprule
        $\alpha$ & budget parameter & $N=\alpha k$ 
        \\
        $\beta$ & size parameter & $r=(1+\beta)k$ 
        \\
        $\gamma$ & oversampling parameter & $l \ge \gamma^2 k$ and $2q+1 \le \frac{\alpha}{\gamma^2}$
        \\
    \bottomrule
    \end{tabular}
\end{table} 

\begin{figure}[!ht]
    \centering
    \includegraphics[width=\textwidth]{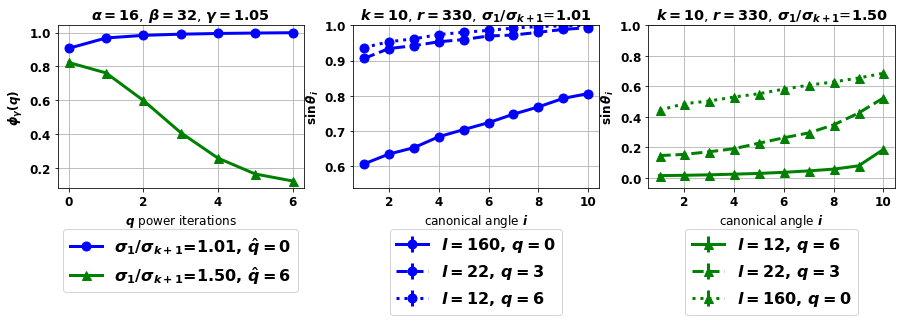}
    \caption{For $k=10$, $\alpha=16$, $\beta=32$, $\gamma=1.05$, the left figure marks $\phi_\gamma(q)$ (\ie, the right-hand side of \Cref{eq:space_agnostic_left} under the fixed budget $N$) with two different spectral gaps ($\wh q = \argmin_{1 \le 2q+1 \le \alpha/\gamma^2} \phi_\gamma\rbr{q}$), while the middle and the right figures demonstrate how the relative magnitudes of canonical angles $\sin\angle_i\rbr{\Ub_k,\wh\Ub_l}$ ($i \in [k]$) under different configurations (\ie, choices of $(l,q)$, showing the averages and ranges of $5$ trials)  align with the trends in $\phi_\gamma(q)$.}
    \label{fig:space-agno-ex_a16-b32-g1}
\end{figure}

\begin{figure}[!ht]
    \centering
    \includegraphics[width=\textwidth]{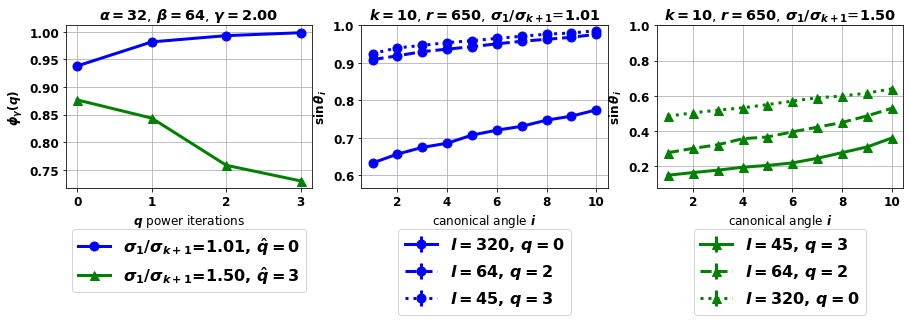}
    \caption{Under the same setup as \Cref{fig:space-agno-ex_a16-b32-g1}, for $k=10$, $\alpha=32$, $\beta=64$, $\gamma=2.00$, the trend in $\phi_\gamma(q)$ also aligns well with that in true canonical angles $\sin\angle_i\rbr{\Ub_k,\wh\Ub_l}$ ($i \in [k]$).}
    \label{fig:space-agno-ex_a32-b64-g2}
\end{figure}

With $k=10$, $\alpha=32$, $\beta=64$, and $\gamma \in \cbr{1.05, 2.00}$, \Cref{fig:space-agno-ex_a16-b32-g1} and \Cref{fig:space-agno-ex_a32-b64-g2} illustrate 
\begin{enumerate*}[label=(\roman*)]
    \item how the balance between oversampling and power iterations is affected by the spectral gap $\sigma_1/\sigma_{k+1}$, and more importantly, 
    \item how \Cref{eq:space_agnostic_example_rhs} unveils the trend in true canonical angles $\sin\angle_i\rbr{\Ub_k,\wh\Ub_l}$ among different configurations $\csepp{(l,q)}{l \ge \gamma^2k, 2q+1 \le \alpha/\gamma^2}$.
\end{enumerate*}

Concretely, both \Cref{fig:space-agno-ex_a16-b32-g1} and \Cref{fig:space-agno-ex_a32-b64-g2} imply that more oversampling (\eg, $q=0$) is preferred when the spectral gap is small (\eg, $\sigma_1/\sigma_{k+1}=1.01$); while more power iterations (\eg, $q=\lfloor \frac{\alpha/\gamma^2-1}{2} \rfloor$) are preferred when the spectral gap is large (\eg, $\sigma_1/\sigma_{k+1}=1.5$). Such trends are both observed in the true canonical angles $\sin\angle_i\rbr{\Ub_k,\wh\Ub_l}$ ($i \in [k]$) and well reflected by $\phi_\gamma\rbr{q}$.
\end{example}

\section{Discussion, Limitations, and Future Directions}\label{sec:discussion}

{We presented} prior and posterior bounds and estimates that can be computed efficiently for canonical angles of the randomized subspace approximation. 
Under moderate multiplicative oversampling, our prior probabilistic bounds are space-agnostic (\ie, independent of the unknown true subspaces), asymptotically tight, and can be computed in linear ($O(\rank\rbr{\Ab})$) time, while casting a clear {guidance} on the balance between oversampling and power iterations for a fixed budget of matrix-vector multiplications.
As corollaries of the prior probabilistic bounds, we introduce a set of unbiased canonical angle estimates that are efficiently computable and applicable to arbitrary choices of oversampling with good empirical concentrations.
In addition to the prior bounds and estimates, we further discuss two sets of posterior bounds that provide deterministic guarantees for canonical angles given the computed low-rank approximations.
With numerical experiments, we compare the empirical tightness of different canonical angle bounds and estimates on various data matrices under a diverse set of algorithmic choices for the randomized subspace approximation. 

As a major limitation of our space-agnostic bounds and estimates, we note that \Cref{thm:space_agnostic_bounds} and \Cref{prop:space_agnostic_estimation} rely crucially on the isotropic property of Gaussian random matrices and therefore are only applicable to Gaussian embeddings. Although Gaussian embedding is one of the most commonly used Johnson-Lindenstrauss transforms (JLT) both theoretically and in practice, in light of the appealing empirical performance~\cite{halko2011finding,martinsson2020randomized,dong2021simpler} of various alternatives (\eg, the subsampled randomized trigonometric transforms~\cite{woolfe2008fast,rokhlin2008fast,tropp2011improved, boutsidis2013srtt} and sparse sign matrices~\cite{meng2013low, nelson2013,woodruff2015sketching,clarkson2017low,tropp2017fixed}), it would be interesting to explore relaxations of the isotropic assumption and extend the analysis to these fast JLTs.

\section*{Acknowledgments}
The work reported was supported by the Office of Naval
Research (N00014-18-1-2354), by the National Science Foundation (DMS-1952735, DMS-2012606), and by the Department of Energy ASCR (DE-SC0022251). The authors wish to thank the editor and the referees for their valuable insights and constructive suggestions.

\bibliographystyle{siamplain}
\bibliography{ref}

\clearpage
\appendix
\section{Technical Lemmas}\label{subapx:technical_proofs}
\begin{lemma}\label{lemma:sample_to_population_covariance}
Let $\xb \in \R^d$ be a random vector with $\E[\xb]=\b{0}$, $\E[\xb\xb^{\top}] = \Sigmab$, and $\overline{\xb} = \Sigmab^{-1/2} \xb$ 
\footnote{In the case where $\Sigmab$ is rank-deficient, we slightly abuse the notation such that $\Sigmab^{-1/2}$ and $\Sigmab^{-1}$ refer to the respective pseudo-inverses.}
being $\rho^2$-subgaussian\footnote{A random vector $\vb \in \R^d$ is $\rho^2$-subgaussian if for any unit vector $\ub \in \mathbb{S}^{d-1}$, $\ub^{\top} \vb$ is $\rho^2$-subgaussian, $\E \sbr{\exp(s \cdot \ub^{\top} \vb)} \leq \exp\rbr{s^2 \rho^2/2}$ for all $s \in \R$.}.
Consider a set of $n$ $\iid$ samples of $\xb$, {$\Xb=[\xb_1,\cdots,\xb_n]^* \in \R^{n \times d}$}, and a diagonal weight matrix $\Wb=\diag\rbr{w_1,\dots,w_n}$ with $w_i>0$ corresponding to each sample $i \in [n]$.
If $n \ge \frac{\tr\rbr{\Wb}^2}{\tr\rbr{\Wb^2}} \ge \frac{20736 \rho^4 d}{\eps^2} + \frac{10368\rho^4 \log(1/\delta)}{\eps^2}$, then with probability at least $1-\delta$,
\begin{align*}
    \rbr{1-\eps} \tr\rbr{\Wb} \Sigmab \aleq \Xb^{\top} \Wb \Xb \aleq \rbr{1+\eps} \tr\rbr{\Wb} \Sigmab
\end{align*}
Concretely, with $\Wb = \Ib_n$, $n = \frac{\tr\rbr{\Wb}^2}{\tr\rbr{\Wb^2}} = \Omega\rbr{\rho^4 d}$, and $\eps=\Theta\rbr{\rho^2 \sqrt{\frac{d}{n}}}$, $(1-\eps) \Sigmab \aleq \frac{1}{n}\Xb^{\top} \Xb \aleq (1+\eps) \Sigmab$ with high probability (at least $1-e^{-\Theta(d)}$).
\end{lemma}

\begin{proof}
We first denote $\Pb_{\Xcal} \triangleq \Sigmab \Sigmab^{\pinv}$ as the orthogonal projector onto the subspace $\Xcal \subseteq \R^d$ supported by the distribution of $\xb$.
With the assumptions $\E[\xb]=\b{0}$ and $\E[\xb\xb^{\top}] = \Sigmab$, we observe that $\E \sbr{\overline{\xb}} = \b{0}$ and $\E\sbr{\overline{\xb} \overline{\xb}^{\top}} = \E \sbr{\rbr{\Sigmab^{-1/2}\xb} \rbr{\Sigmab^{-1/2}\xb}^{\top}} = \Pb_{\Xcal}$.
Given the sample set $\Xb$ of size $n \gg \rho^4\rbr{d+\log(1/\delta)}$ for any $\delta \in (0,1)$, we let
\begin{align*}
    \Ub = \frac{1}{\tr\rbr{\Wb}} \sum_{i=1}^n w_i \rbr{\Sigmab^{-1/2}\xb} \rbr{\Sigmab^{-1/2}\xb}^{\top} - \Pb_{\Xcal}.
\end{align*}
Then the problem can be reduced to showing that, for any $\eps>0$, with probability at least $1-\delta$, $\norm{\Ub}_2 \leq \eps$.
For this, we leverage the $\eps$-net argument as follows.

For an arbitrary $\vb \in \Xcal \cap\ \mathbb{S}^{d-1}$, we have
\begin{align*}
    \vb^{\top} \Ub \vb = 
    &\frac{1}{\tr\rbr{\Wb}} \sum_{i=1}^n 
    w_i \rbr{\vb^{\top} \rbr{\Sigmab^{-1/2}\xb} \rbr{\Sigmab^{-1/2}\xb}^{\top} \vb - 1} 
    \\= 
    &\frac{1}{\tr\rbr{\Wb}} \sum_{i=1}^n 
    w_i \rbr{\rbr{\vb^{\top} \overline{\xb}_i}^2 - 1},
\end{align*}
where, given $\overline{\xb}_i$ being $\rho^2$-subgaussian, 
$\vb^{\top} \overline{\xb}_i$ is $\rho^2$-subgaussian. 
Since 
\begin{align*}
    \E \sbr{\rbr{\vb^{\top} \overline{\xb}_i}^2} = \vb^{\top} \E \sbr{\overline{\xb}_i \overline{\xb}_i^{\top}} \vb = 1,
\end{align*}
we know that $\rbr{\vb^{\top} \overline{\xb}_i}^2-1$ is $16\rho^2$-subexponential\footnote{We abbreviate $\rbr{\nu,\nu}$-subexponential (\ie, recall that a random variable $X$ is $\rbr{\nu,\alpha}$-subexponential if $\E\sbr{\exp\rbr{s X}} \le \exp\rbr{s^2 \nu^2/2}$ for all $\abbr{s} \le 1/\alpha$) simply as $\nu$-subexponential.}.
With $\beta_i \dfeq \frac{w_i}{\tr\rbr{\Wb}}$ for all $i \in [n]$ such that {$\betab = \sbr{\beta_1,\cdots,\beta_n}^*$}, we recall Bernstein's inequality \cite[Theorem 2.8.2]{vershynin2018high}\cite[Section 2.1.3]{wainwright2019high},
\begin{align*}
    &\PP\sbr{\abbr{\vb^{\top} \Ub \vb} = \abbr{\sum_{i=1}^n \beta_i \cdot \rbr{\rbr{\vb^{\top} \overline{\xb}_i}^2-1}} > t} \\
    &\leq 
    2 \exp \rbr{-\frac{1}{2} 
    \min\rbr{\frac{t^2}{\rbr{16 \rho^2}^2 \nbr{\betab}_2^2}, 
    \frac{t}{16 \rho^2 \nbr{\betab}_\infty}}},
\end{align*}
where $\nbr{\betab}_2^2 = \frac{\tr\rbr{\Wb^2}}{\tr\rbr{\Wb}^2}$ and $\nbr{\betab}_\infty = \frac{\max_{i \in [n]} w_i}{\tr\rbr{\Wb}}$.

Let $N \subset \Xcal \cap\ \mathbb{S}^{d-1}$ be an $\eps_1$-net such that $\abbr{N} = \rbr{1+\frac{2}{\eps_1}}^d$. 
Then for some $0 < \eps_2 \leq 16 \rho^2 \frac{\nbr{\betab}_2^2}{\nbr{\betab}_\infty}$, by the union bound,
\begin{align*}
    \PP \sbr{\underset{\vb \in N}{\max}: \abbr{\vb^{\top} \Ub \vb} > \eps_2} 
    \leq \ 
    & 2 \abbr{N} \exp \rbr{-\frac{1}{2}
    \min\rbr{\frac{\eps_2^2}{\rbr{16 \rho^2}^2 \nbr{\betab}_2^2}, 
    \frac{\eps_2}{16 \rho^2 \nbr{\betab}_\infty}} } 
    \\ \leq \ 
    & \exp \rbr{d\log\rbr{1+\frac{2}{\eps_1}} -\frac{1}{2} \cdot \frac{\eps_2^2}{\rbr{16 \rho^2}^2 \nbr{\betab}_2^2} } \le \delta
\end{align*}
whenever $\frac{1}{\nbr{\betab}_2^2} \ge \frac{\tr\rbr{\Wb}^2}{\tr\rbr{\Wb^2}} = 2 \rbr{\frac{16 \rho^2}{\eps_2}}^2 \rbr{d\log\rbr{1+\frac{2}{\eps_1}} + \log\frac{1}{\delta}}$ where $1 < \frac{\tr\rbr{\Wb}^2}{\tr\rbr{\Wb^2}} \le n$. 

Now for any $\vb \in \Xcal \cap\ \mathbb{S}^{d-1}$, there exists some $\vb' \in N$ such that $\norm{\vb - \vb'}_2 \leq \eps_1$. 
Therefore,
\begin{align*}
    \abbr{\vb^{\top} \Ub \vb} 
    \ = \
    & \abbr{\vb'^{\top} \Ub \vb' + 
    2 \vb'^{\top} \Ub \rbr{\vb - \vb'} + 
    \rbr{\vb - \vb'}^{\top} \Ub \rbr{\vb - \vb'} }
    \\ \leq \ 
    & \rbr{ \underset{\vb \in N}{\max}: \abbr{\vb^{\top} \Ub \vb} } + 
    2 \norm{\Ub}_2 \norm{\vb'}_2 \norm{\vb-\vb'}_2 + 
    \norm{\Ub}_2 \norm{\vb-\vb'}_2^2
    \\ \leq \ 
    & \rbr{ \underset{\vb \in N}{\max}: \abbr{\vb^{\top} \Ub \vb} } + 
    \norm{\Ub}_2
    \rbr{2 \eps_1 + \eps_1^2}.
\end{align*}
Taking the supremum over $\vb \in \mathbb{S}^{d-1}$, with probability at least $1-\delta$,
\begin{align*}
    \underset{\vb \in \Xcal \cap\ \mathbb{S}^{d-1}}{\max}: \abbr{\vb^{\top} \Ub \vb}
    = 
    \norm{\Ub}_2
    \leq 
    \eps_2 + \norm{\Ub}_2 \rbr{2 \eps_1 + \eps_1^2},
    \qquad
    \norm{\Ub}_2
    \leq 
    \frac{\eps_2}{2 - \rbr{1+\eps_1}^2}.
\end{align*}
With $\eps_1 = \frac{1}{3}$, we have $\eps = \frac{9}{2}\eps_2$.

Overall, if $n \ge \frac{\tr\rbr{\Wb}^2}{\tr\rbr{\Wb^2}} \ge \frac{1024 \rho^4 d}{\eps_2^2} + \frac{512 \rho^4}{\eps_2^2}\log\frac{1}{\delta}$, then with probability at least $1-\delta$, we have $\norm{\Ub}_2 \leq \eps$.

As a concrete instance, when $\Wb=\Ib_n$ and $n = \frac{\tr\rbr{\Wb}^2}{\tr\rbr{\Wb^2}} \ge 9^2 \cdot 1025 \cdot \rho^4 d$, by taking $\eps_2 = \sqrt{\frac{1025 \rho^4 d}{n}}$, we have $\nbr{\Ub}_2 \le \frac{1}{2} \sqrt{\frac{9^2 \cdot 1025 \cdot \rho^4 d}{n}}$ with high probability (at least $1-\delta$ where $\delta = \exp\rbr{-\frac{d}{512}}$).
\end{proof}

\begin{lemma}[Cauchy interlacing theorem]\label{lemma:Cauchy_interlacing_theorem}
Given an arbitrary matrix $\Ab \in \C^{m \times n}$ and an orthogonal projection $\Qb \in \C^{n \times k}$ with orthonormal columns, for all $i=1,\dots,k$,
\begin{align*}
    \sigma_{i}\rbr{\Ab \Qb} \leq \sigma_{i}\rbr{\Ab}.
\end{align*}
\end{lemma}

\begin{proof}[Proof of \Cref{lemma:Cauchy_interlacing_theorem}]
Let $\csepp{\vb_j \in \C^k}{j=1,\dots,k}$ be right singular vectors of $\Ab\Qb$. By the min-max theorem (\cf \cite{golub2013} Theorem 8.6.1), 
\begin{align*}
    \sigma_{i}\rbr{\Ab \Qb}^2
    = &\min_{\xb \in \spn\cbr{\vb_1,\dots,\vb_i}\setminus\rbr{\b0}} \frac{\xb^\top \Qb^\top \Ab^\top \Ab \Qb \xb}{\xb^\top \xb} \\
    \leq &\max_{\dim\rbr{\Vcal}=i} \min_{\xb \in \Vcal\setminus\rbr{\b0}} \frac{\xb^\top \Ab^\top \Ab \xb}{\xb^\top \xb} = \sigma_i\rbr{\Ab}^2.
\end{align*}  
\end{proof}

\begin{lemma}[\cite{horn_johnson_2012} (7.3.14)]\label{lemma:individual_sval_of_product}
For arbitrary matrices $\Ab,\Bb \in \C^{m \times n}$,
\begin{align}\label{eq:lemma_individual_sval_of_product}
    \sigma_i\rbr{\Ab \Bb^*} \le \sigma_i\rbr{\Ab} \sigma_j\rbr{\Bb}
\end{align} 
for all $i \in [\rank\rbr{\Ab}]$, $j \in [\rank\rbr{\Bb}]$ such that $i+j-1 \in [\rank\rbr{\Ab\Bb^*}]$.
\end{lemma}

\begin{lemma}[\cite{horn_johnson_2012} (7.3.13)] \label{lemma:individual_sval_of_sum}
For arbitrary matrices $\Ab,\Bb \in \C^{m \times n}$,
\begin{align}\label{eq:lemma_individual_sval_of_sum}
    \sigma_{i+j-1}\rbr{\Ab + \Bb} \le \sigma_i\rbr{\Ab} + \sigma_j\rbr{\Bb}
\end{align}    
for all $i \in [\rank\rbr{\Ab}]$, $j \in [\rank\rbr{\Bb}]$ such that $i+j-1 \in [\rank\rbr{\Ab+\Bb}]$.
\end{lemma} 

\section{Supplementary Experiments: Lower Space-agnostic Bounds}\label{subapx:sup_exp_upper_lower_space_agnostic}

In this section, we visualize and compare the upper and lower bounds in \Cref{thm:space_agnostic_bounds} under the sufficient multiplicative oversampling regime (\ie, $l = 4k$. Recall that $k < l < r = \rank\rbr{\Ab}$ where $k$ is the target rank, $l$ is the oversampled rank, and $r$ is the full rank of the matrix $\Ab$).

\begin{figure}[!ht]
    \centering
    \includegraphics[width=\linewidth]{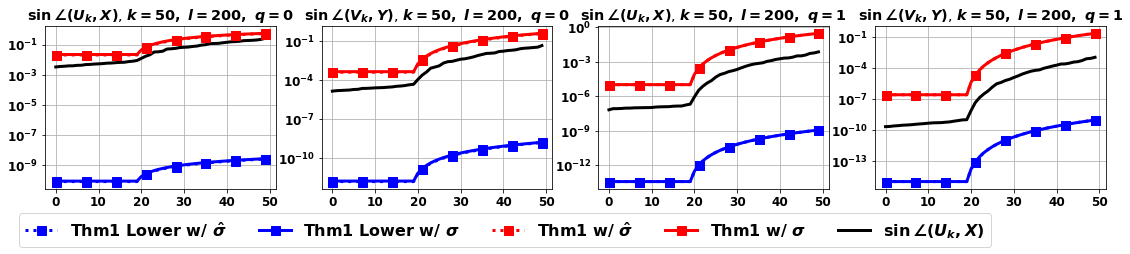}
    \caption{Synthetic Gaussian with the slower spectral decay. $k=50$, $l=200$, $q=0,1$.}
    \label{fig:lower_Gaussian-poly1-kl_k50_l200}
\end{figure}

\begin{figure}[!ht]
    \centering
    \includegraphics[width=\linewidth]{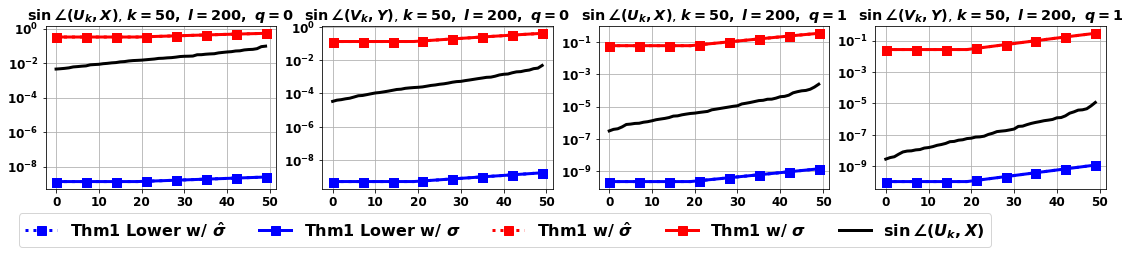}
    \caption{Synthetic Gaussian with the faster spectral decay. $k=50$, $l=200$, $q=0,1$.}
    \label{fig:lower_Gaussian-exp-kl_k50_l200}
\end{figure}

\begin{figure}[!ht]
    \centering
    \includegraphics[width=\linewidth]{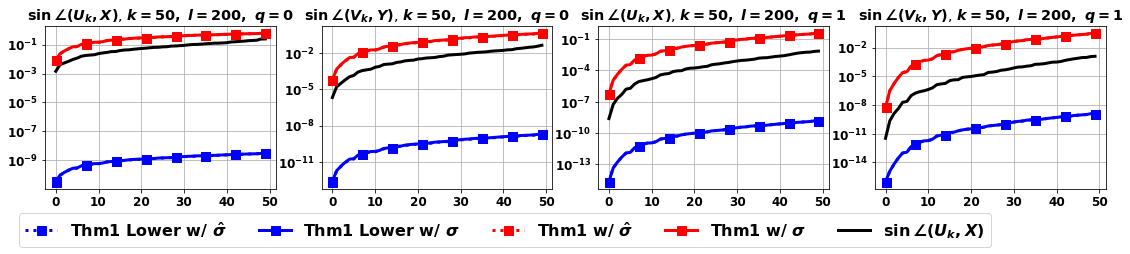}
    \caption{SNN with $r_1=20$, $a=1$. $k=50$, $l=200$, $q=0,1$.}
    \label{fig:lower_SNN-m500n500r20a1-kl_k50_l200}
\end{figure}

\begin{figure}[!ht]
    \centering
    \includegraphics[width=\linewidth]{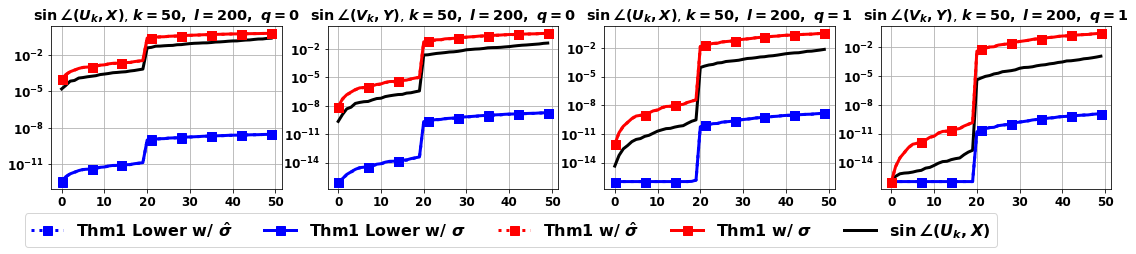}
    \caption{SNN with $r_1=20$, $a=100$. $k=50$, $l=200$, $q=0,1$.}
    \label{fig:lower_SNN-m500n500r20a100-kl_k50_l200}
\end{figure}

\begin{figure}[!ht]
    \centering
    \includegraphics[width=\linewidth]{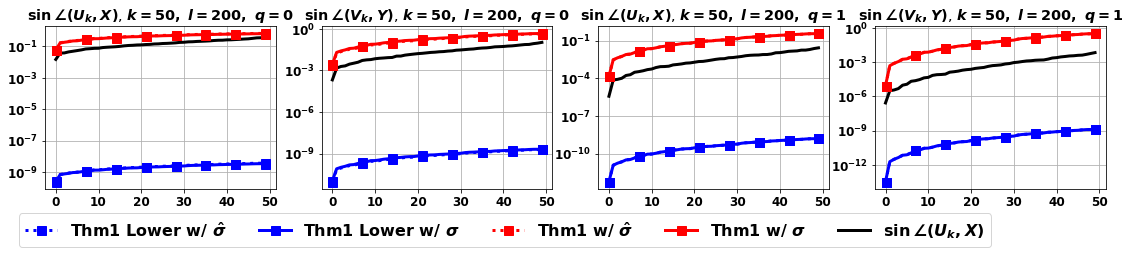}
    \caption{$800$ randomly sampled images from the MNIST training set. $k=50$, $l=200$, $q=0,1$.}
    \label{fig:lower_mnist-train800-kl_k50_l200}
\end{figure}

With the same set of target matrices described in \Cref{subsec:target_matrix}, from \Cref{fig:lower_Gaussian-poly1-kl_k50_l200} to \Cref{fig:lower_mnist-train800-kl_k50_l200},
\begin{enumerate}
    \item {
        Red lines and dashes (\b{Thm1 w/ $\sigmab$ and $\wh\sigmab$})
    } show the upper bounds in \Cref{eq:space_agnostic_left,eq:space_agnostic_right} evaluated with the true (lines) and approximated (dashes) singular values, $\Sigmab$, and $\wt\Sigmab$, respectively, where we simply ignore tail decay and suppress constants for the distortion factors and set 
    \begin{align*}
        \eps_1 = \sqrt{\frac{k}{l}}
        \quad\text{and}\quad
        \eps_2 = \sqrt{\frac{l}{r-k}}.
    \end{align*}
    
    \item {
        Blue lines and dashes (\b{Thm1 Lower w/ $\sigmab$ and $\wh\sigmab$})
    } present the lower bounds in \Cref{eq:space_agnostic_lower_left,eq:space_agnostic_lower_right} evaluated with $\Sigmab$ and $\wt\Sigmab$, respectively, and slightly larger constants associated with the distortion factors 
    \begin{align*}
        \eps'_1 = 2\sqrt{\frac{k}{l}}
        \quad\text{and}\quad
        \eps'_2 = 2\sqrt{\frac{l}{r-k}}.
    \end{align*}
\end{enumerate}

The numerical observations imply that the empirical validity of lower bounds requires more aggressive oversampling than that of upper bounds. In particular, we recall from \Cref{subsec:experiment_canonical_angle} that $l \ge 1.6k$ is usually sufficient for the upper bounds to hold numerically. In contrast, the lower bounds generally require at least $l \ge 4k$, with slightly larger constants associated with the distortion factors $\eps_1 = \Theta\rbr{\sqrt{k/l}}$ and $\eps_2 = \Theta\rbr{\sqrt{l/\rbr{r-k}}}$.

\end{document}